\documentclass[11pt,onecolumn]{IEEEtran}

\usepackage{amsmath,amsthm,amssymb}
\usepackage{graphicx,epstopdf,epsfig}
\usepackage{url,framed}
\usepackage{bm}
\usepackage{palatino}
\usepackage{multirow}
\usepackage[caption=false,font=footnotesize]{subfig}
\usepackage[linesnumbered,ruled]{algorithm2e}
\usepackage{xcolor}
\usepackage{array}
\usepackage{float}

\SetKwIF{If}{ElseIf}{Else}{if}{}{else if}{else}{end if}
\SetKwFor{While}{while}{}{end while}
\SetKwRepeat{Do}{do}{while}
\SetKw{KwGoTo}{go to}

\def \x {\mathbf{x}}
\def \a {\mathbf{a}}
\def \t {\mathbf{t}}
\def \q {\mathbf{q}}
\def \e {\mathbf{e}}
\def \f {\mathbf{f}}
\def \del {\bm{\delta}}
\def \bO {\mathbf{O}}
\def \bX {\mathbf{X}}
\def \bD {\mathbf{D}}
\def \bI {\mathbf{I}}
\def \bJ {\mathbf{J}}
\def \bB {\mathbf{B}}
\def \bA {\mathbf{A}}
\def \bZ {\mathbf{Z}}
\def \bG {\mathbf{G}}
\def \bC {\mathbf{C}}
\def \bH {\mathbf{H}}

\def \bV {\mathbf{V}}
\def \bU {\mathbf{U}}

\def \Lam {\bm{\Lambda}}
\def \cV {\mathcal{V}}
\def \cE {\mathcal{E}}
\def \cG {\mathcal{G}}
\def \cC {\mathcal{C}}
\def \cS {\mathcal{S}}
\def \cA {\mathcal{A}}
\def \cS {\mathcal{S}}
\def \cN {\mathcal{N}}
\def \R {\mathbb{R}}

\def \O {\mathbb{O}}
\def \S {\mathbb{S}}
\def \cR {\mathcal{R}}
\def \cQ {\mathcal{Q}}
\def \cT {\mathcal{T}}
\def \cS {\mathcal{S}}

\def \cO {\mathcal{O}}
\def \cX {\mathcal{X}}

\newtheorem{theorem}{Theorem}[section]

\newtheorem{proposition}[theorem]{Proposition}
\newtheorem{definition}[theorem]{Definition}
\newtheorem{problem}[theorem]{Problem}
\newtheorem{assumption}[theorem]{Assumption}
\newtheorem{conjecture}[theorem]{Conjecture}

\begin{document}

\title{On a Registration-Based Approach to Sensor Network Localization}

\author{R.~Sanyal,
        M.~Jaiswal,
        and~K.~N.~Chaudhury

\thanks{The authors were supported by a Startup Grant from IISc Bangalore and an 
EMR Grant SERB/F/6047/2016-2017 from Department of Science and  Technology, Government of India. The second author was supported by a NBHM Postdoctoral Fellowship from Department of Atomic Energy, Government of India. Address: Department of Electrical Engineering, Indian Institute of Science, India. Correspondence: \{rajatsanyal,monika,kunal\}@ee.iisc.ernet.in.
}}

\maketitle

\begin{abstract}
We consider a registration-based approach for localizing sensor networks from range measurements. 
This is based on the assumption that one can find overlapping \textit{cliques} spanning the network. 
That is, for each sensor, one can identify geometric neighbors for which all inter-sensor ranges are known.
Such cliques can be  efficiently localized using multidimensional scaling. However, since each clique is localized in some local coordinate system, we are required to \textit{register} them in a global coordinate system. In other words, our approach is based on transforming the localization problem into a problem of registration. In this context, the main contributions are as follows. First, we describe an efficient method for partitioning the network into overlapping cliques. 
Second, we study the problem of registering the localized cliques, and formulate a necessary \textit{rigidity} condition  for uniquely recovering the global sensor coordinates. 
In particular, we present a method for efficiently testing rigidity, and a proposal for augmenting the partitioned network to enforce rigidity.
A recently proposed semidefinite relaxation of global registration is used for registering the cliques. We present simulation results on random and structured sensor networks to demonstrate that the proposed method compares favourably with state-of-the-art methods in terms of run-time, accuracy, and scalability. 
\end{abstract}

\begin{IEEEkeywords}
Sensor networks, localization, scalability, rigidity, clique, multidimensional scaling, semidefinite programming.
\end{IEEEkeywords}

\section{Introduction}

Recent developments in wireless communication and micro-electro-mechanics have proliferated the deployment of wireless sensor networks (WSN) \cite{YMG2008}.
A typical WSN may consist of few tens to thousands of nodes. Each node is a low-power device equipped with transducers, power supply, memory, processor, radio transmitter, and actuators.
A global positioning system (GPS) is often installed on some of the nodes. Such nodes are referred to as \textit{anchor} nodes.
However, only a small fraction of the nodes are equipped with GPS to minimize weight and power consumption. In this paper, we will use \textit{sensor} to specifically refer to a node that does not have a GPS, while the term \textit{node} will be used for both sensors and anchors. WSNs are mostly deployed  in remote locations, and nodes have limited memory capacity, so wireless transmitters are used to transfer the sensor data to base stations. Due to power constraints, two nodes can communicate if and only if the inter-node distance is within some \textit{radio range}, which we will denote by $r$ \cite{YMG2008}. We would like to note that although GPS modules are getting cheaper, deploying them in large scale would still be costly. Moreover, GPS comes with its own limitations \cite{YMG2008}. To calculate the position of a sensor using GPS alone, at least four line-of-sights (with satellites) are required. This might not be viable in case of bad weather. Furthermore, for underwater surveys and mining applications, it is not even feasible to have line-of-sights. In fact, in applications where the position information is crucial, localization algorithms can be used to back up GPS positioning.

To meaningfully interpret the sensor data, one requires the locations of the sensors. A central problem in this regard is to estimate the sensor locations from the inter-sensor distances and the anchor locations. This problem is referred to as sensor network localization (SNL) \cite{SY2007,MFA2007}. To set up the mathematical description of SNL, we introduce some notations that will be follow throughout the paper. Assume that we have a total of $N$ sensors  and $K$ anchors. We label the sensors using $\cS=\{1,\ldots,N\}$, the anchors using $\cA=\{N+1,\ldots,N+K\}$, and $\cN = \cS  \cup \cA$ denotes the nodes in general. Let 
\begin{equation}
\label{nodes}
\cX_s=\{\bar{\x}_i : i \in \cS\} \quad \text{and} \quad \cX_a=\{\bar{\a}_k : k \in \cA\}
\end{equation}
denote the sensor and anchor locations. We assume $\bar{\x}_i$ and $\bar{\a}_k$ to be in $\mathbb{R}^d$, where $d$ is typically $2$ or $3$ \cite{YMG2008,SY2007}.  

The distance between two nodes $i$ and $j$ (that are within the radio range $r$) can be calculated using 
different techniques, such as the received signal strength or the time of arrival \cite{MFA2007}. A \textit{measurement graph} $\cG$ is used to encode the distance information \cite{SXG2015,CLS2012}. Particularly, $\cG = (\cV,\cE)$, where $\cV(\cG)=\cN$, and  $(i,j) \in \cE(\cG)$ if and only if the distance between the $i$-th and the $j$-th node is known. The problem is to compute the unknown sensor locations $\cX_s$ from the measured distances and the anchor locations $\cX_a$. We make the standard assumption that the anchor locations are noise-free \cite{SXG2015,SL2014}.
\subsection{Optimization Algorithms}

The decision version of the SNL problem is known to be computationally intractable \cite{Saxe1979}. The presence of noise makes the problem even more challenging in practice. Nonetheless, several methods have been proposed that can compute approximate solutions. A survey of the literature on SNL is beyond the scope of this paper. Instead, we will focus on some of the recent optimization methods that are related to the present work. We refer the interested reader to \cite{MFA2007} for a survey of algorithms that are not based on optimization. 

The simplest optimization framework for SNL is that of strain minimization \cite{BLTYW2006}. In this approach, the sensor locations $\x_1, \dots , \x_N$ are obtained by minimizing the \textit{strain} function
\begin{equation}
\label{SNL_strain}
\!\sum_{(i,j) \in \cE } \! \left( \lVert \x_i - \x_j \rVert^2 - d_{ij}^2 \right)^2\! +\!  \sum_{(i,k) \in \cE} \! \left( \lVert \x_i - \a_k \rVert^2 - d_{ik}^2 \right)^2.
\end{equation}
In \eqref{SNL_strain}, the indices $i,j$ are reserved for $\cS$, and the index $k$ for $\cA$. Unfortunately, it is  difficult to compute the global minimum of \eqref{SNL_strain} since it is non-convex in the variables \cite{SY2007}. In this regard, several approximation algorithms based on convex programming have been proposed, which can provably compute the global minimum under certain conditions. Based on the computing paradigm, one can broadly classify these as centralized and distributed algorithms. 

Centralized algorithms employ a server to store the transmitted range measurements, based on which the sensor locations are computed. It was observed in \cite{DPG2001} that the distance bounds in SNL can be posed as semidefinite constraints. Later, in the seminal paper \cite{BLTYW2006}, the authors showed how \eqref{SNL_strain} can be approximated using a convex semidefinite program (SDP). The main advantage of posing SNL as a convex program is that we can find the global minimizer of the problem independent of the initialization. The flip side, however, is that standard SDP solvers (e.g., SeDuMi \cite{Sturm1999}) are memory and computation intensive, and hence cannot be scaled to large-sized problems. For example, the SDP-based algorithm in \cite{BLTYW2006} can scale only up to a few hundred nodes \cite{WZYB2008}. To improve the scalability,  a further edge-based relaxation of  \cite{BLTYW2006} was proposed in \cite{WZYB2008}. While the relaxation can indeed scale up to $8000$ nodes, its performance is nevertheless inferior to that of the original SDP for medium-sized problems. 

On the other hand, distributed algorithms divide the processing over the nodes. As a result, they exhibit better scalability compared to centralized methods. The main drawback is that they suffer from error propagation \cite{MFA2007}. Moreover, distributed methods such as the ones in \cite{SXG2015,SL2014,ADPV2012,GTSC2013} can operate only in the presence of anchors (which might not be available, e.g., in indoor WSN). The distributed algorithm in \cite{ADPV2012} that can handle million sensors without any significant communication overhead. However, the localization accuracy of this method is conditioned on a good initialization. More recently, distributed methods based on convex programming have been proposed in \cite{SXG2015,SL2014,GTSC2013}. In particular, a distributed algorithm based on the alternating direction method of multipliers was proposed in \cite{SL2014}. However, as reported in \cite{SXG2015}, the approach is computationally demanding since each node is required to solve an SDP per iteration, and also the communication overhead is significant. A distributed algorithm that is cheaper and requires a smaller communication overhead was later proposed in \cite{SXG2015}. One shortcoming of the latter method is that it requires the sensors to be in the convex hull of the anchors, which is difficult to guarantee in practice. 

The present work was motivated by a class of centralized algorithms that use divide-and-conquer approaches to improve scalability \cite{CLS2012,CKS2015b,LT2009,ZLGG2010}. The general mechanism is to partition $\cG$ into overlapping subgraphs, localize each subgraph using the induced distances, and finally register the subgraphs. The idea is to construct subgraphs that are denser than the original graph. Moreover, the smaller graphs can be efficiently localized. The algorithms essentially differ on how each subproblem is solved. For example, a Cuthill-McKee-type permutation is used  in \cite{LT2009} to partition $\cG$. In \cite{CLS2012,ZLGG2010}, $\cG$ is partitioned using neighborhood subgraphs. To improve the localization, rigid subgraphs \cite{CW2009} are extracted from each neighborhood subgraph in \cite{CLS2012}. Recently, recursive spectral clustering was used in \cite{CKS2015b}. We note that the subgraphs obtained using the graph partitioning in \cite{CKS2015b,LT2009}  are not guaranteed to be rigid \cite{Laman1970,GHT2010}, and hence can result in poor localization. In fact, a few poorly localized subgraphs can adversely affect the overall registration. The algorithms in \cite{LT2009,ZLGG2010} register the subgraphs in a sequential fashion, and inevitably suffer from error propagation. Recently, a least-square method was proposed in \cite{CKS2015a} that can register the subgraphs in a globally-consistent manner. In particular, it was demonstrated in \cite{CKS2015a} that global registration can successfully operate in adversarial situations where sequential methods fail. In this regard, we note that a \textit{lateration} condition was introduced in \cite{CKS2015a} that can guarantee exact recovery in the noise-free setting. However, there is no known efficient algorithm for testing lateration.

\subsection{Contributions}

We propose a divide-and-conquer algorithm building on the ideas in \cite{CKS2015b,CKS2015a,SC2016}. In particular, we address the following issues that emerged from this line of work: testing and ensuring that each subgraph is rigid, formulating a testable condition for recovering the sensor coordinates, and developing a scalable algorithm for registering the localized subgraphs. In this context, the contributions are as follows:

(i) To bypass the rigidity issue associated with the localization of each subgraph, we propose to use cliques. Cliques are trivially rigid and can be efficiently localized using multidimensional scaling \cite{BG2003}. However, given that finding cliques in a large graph is challenging, we first partition $\cG$ into neighborhood graphs \cite{CLS2012,ZLGG2010}. We then extract a clique from each neighborhood graph using the algorithm from \cite{BG2015}. Finally, we augment the vertices of a clique to expand it into a maximal clique. We experimentally demonstrate that the complete process is fast for both random and structured geometric graphs. 

(ii) We study the problem of registering a system of localized cliques. In particular, we establish a \textit{rigidity} condition that is necessary for recovering the original sensor coordinates. The proposed condition can be efficiently tested simply by computing the maximum flow between the vertices of an appropriate graph. We present supporting examples to conjecture that the proposed rigidity condition is also sufficient for exact recovery. Moreover, we demonstrate using numerical examples that  the registration performance can be improved in the noisy setting by enforcing the rigidity condition.

We note that a registration-based approach for anchorless SNL was earlier proposed in \cite{SC2016} that uses cliques and cMDS. In the present work, we focus on anchored SNL (though the method can also be used for anchorless SNL). Moreover, we consider a different clique exploration process. Importantly, we investigate the rigidity problem associated with registration which was not discussed in \cite{SC2016}.
\subsection{Organization}

The rest of the paper is organized as follows.  In Section \ref{RP}, we propose a rigidity criteria for the registration problem, and explain how this can be tested efficiently. The proposed graph partitioning is described in Section \ref{GP} keeping the registration problem in mind. Classical multidimensional scaling is reviewed in Section \ref{cMDS}, which is used to localize the cliques. The registration algorithm is described in Section \ref{ADMM}. Experimental results and comparisons are provided in Section \ref{exp}.
\section{The Rigidity Problem}
\label{RP}

\begin{figure*}
\center
\includegraphics[width=0.7\linewidth]{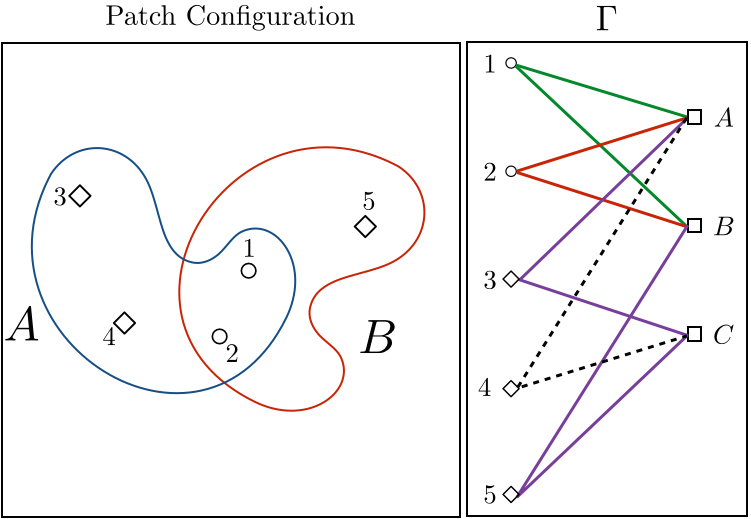}
\caption{A configuration of two patches, two sensors and three anchors (left) and its correspondence graph (right). Circles, diamonds, and squares are used to represent the sensors, anchors, and patches; note that $C$ is the anchor patch. It is clear that this configuration is rigid in two dimensions (see text for precise definitions). We have marked the edges of the three disjoint paths between patch vertices $A$ and $B$ using different colors (right). See text for comments.}
\label{PC1}
\end{figure*}

We first study the fundamental problem of \textit{rigidity} whose resolution will be useful during the graph partitioning phase in Section \ref{GP}. This problem is also relevant for other divide-and-conquer approaches \cite{CLS2012,CKS2015b,ZLGG2010,CKS2015a}, where a system of point sets are required to be registered. More precisely, consider the sensors $\cX_s$ and the anchors $\cX_a$ in \eqref{nodes}, and subsets $\cC_1,\ldots,\cC_M \subset \cN$. Following \cite{CLS2012,CKS2015b}, we will refer to each $\cC_i$ as a \textit{patch}. Moreover, we create an additional patch $\cC_{M+1}$ consisting solely of the anchors $\cA$. Assume that the points in each patch have been derived from the respective points in $\cX_s \cup \cX_a$ via a rigid transform. 
Let
\begin{equation}
 \label{LocGlob1}
\bar{\x}_k =\cR_i(\x_{k,i}) = \bO_i \x_{k,i} + \t_i  \qquad (k \in \cC_i \backslash \cA),
\end{equation}
and
\begin{equation}
 \label{LocGlob2}
\bar{\a}_l =\cR_i(\bar{\a}_l) \qquad (l \in \cC_i \cap \cA),
\end{equation}
where $\x_{k,i}$ is the coordinate of the $k$-th point in the $i$-th patch, and $\cR_i=(\bO_i,\t_i)$ is the rigid transform associated with the $i$-th patch, where the orthogonal matrix $\bO_i$ represents rotation (or reflection) and $\t_i$ is the translation component. We will refer to the $(\x_{k,i})$'s as the \textit{patch coordinates}. The patches and the patch coordinates together form a \textit{configuration}. The \textit{registration} problem is one of determining the unknown $\cX_s$ from the given configuration. 
\begin{problem}[Registration]
\label{prob1}
Find $\x_1, \dots , \x_N$ and rigid transforms $\cQ_1, \dots , \cQ_M$ such that, for $1 \leq i \leq M$,
\begin{equation*}
\x_k = \cQ_i (\x_{k,i}) \quad \text{and}  \quad \bar{\a}_l = \cQ_i (\bar{\a}_l),
\end{equation*}
where $k \in \cC_i \setminus \cA$ and $l \in \cC_i \cap \cA$.
\end{problem}
In the noiseless setting, the solution (points and transforms) sought above exists trivially, namely, the ground truth $\x_k = \bar{\x}_k$ and $\cQ_i=\cR_i$.
The rigidity problem is to determine whether the solution is unique (upto a global rigid transformation).
\begin{problem}[Uniqueness]
\label{prob2} 
Determine whether Problem \ref{prob1} have a unique solution up to a rigid transform. That is, if $\x_1, \dots , \x_N$  is a solution of Problem \ref{prob1}, then is it necessary that for some rigid transform $\cR$, $\x_k = \cR ( \bar{\x}_k)$ and $\bar{\a}_l = \cR (\bar{\a}_l)$, where $k \in \cS$ and $l \in \cA$?
\end{problem}
A set of points in $\R^d$ is said to be \textit{non-degenerate} if their affine span is $\R^d$. Clearly, the cardinality of such points must  be $d+1$ or more. For example, three points are non-degenerate in two-dimensions if and only if they are not collinear. We note that the transforms $\cQ_1, \dots , \cQ_M$ in Problem \ref{prob1} are latent variables and do not appear in the Problem \ref{prob2}.
\begin{definition}[Rigidity]
\label{def_rigidity}
A configuration is said to be \textit{rigid} in $\mathbb{R}^d$ (or simply \textit{rigid}) if the solution is unique in the sense of Problem \ref{prob2}; otherwise, the configuration is said to be \textit{flexible}. 
\end{definition}
To provide geometric insights to the rigidity problem, we consider simple instances of rigid and flexible configurations in Figures \ref{PC1} and \ref{PC3}. In particular, we wish to highlight the importance of overlaps among patches in determining rigidity. In Figure \ref{PC1}, patches $A$ and $B$ share two sensors. The patches can be reflected along the line joining sensors $1$ and $2$, but due to the presence of anchors in $A$ and $B$, reflection is ruled out. The configuration is thus rigid  in the sense of Definition \ref{def_rigidity}. On the other hand, the configuration in Figure \ref{PC3} is flexible since patches $C$ and $D$ can be reflected along the red dotted line.  
\begin{figure}
\center
\includegraphics[width= 0.7\linewidth]{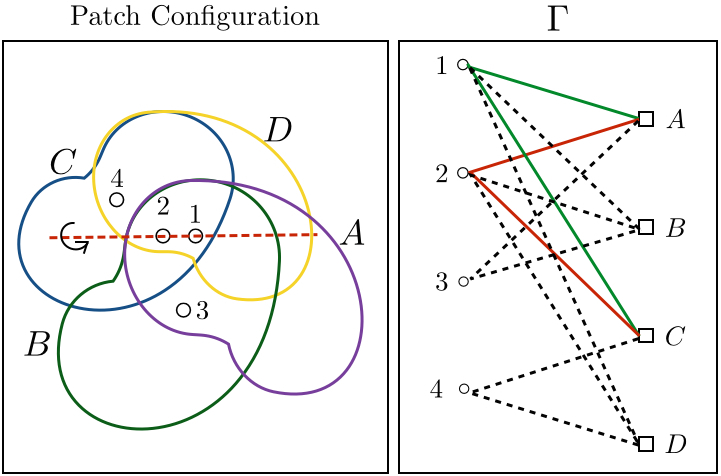}
\caption{Example of a configuration with four patches and four sensors (left). Holding patches $A$ and $B$ fixed, we can reflect patches $C$ and $D$ along the red dotted line. Thus, the configuration is not rigid. Notice that there are just two $\cV_1(\Gamma)$-disjoint paths between patch vertices $A$ and $C$; we have marked the edges of these paths with solid lines (right). See text for comments.}
\label{PC3}
\end{figure}
The above concepts and definitions were motivated by the rigidity aspects of the SNL problem \cite{SY2007}, and more generally, the distance-geometry problem \cite{Laman1970,GHT2010}. Here, the problem is to determine if the available distance measurements uniquely define the sensor locations (modulo a rigid transform which leaves the distances unchanged). A fundamental result in this regard is that, if the original sensor locations are \textit{generic} \cite{GHT2010}, then the uniqueness problem can be completely resolved using just the measurement graph \cite{CW2009,Laman1970,GHT2010}. Our present objective is to come up with similar results for Problem \ref{prob2}. At this point, we wish to emphasize that rigidity theory is solely concerned with exact measurements \cite{GHT2010,ASNF2010}. The point is that the combinatorial structure of the problem should, in principle, be able to guarantee exact recovery of the ground truth when the measurements are perfect. The design of an algorithm that can provably recover the ground truth is however a completely different topic.
We will present some representative examples in Section \ref{exp}, which suggest that rigidity can also help improve the algorithmic performance in the noisy setting. We will assume the following in the rest of the discussion.
\begin{assumption}[Non-degeneracy]
\label{assumption}
There are at least $d+1$ non-degenerate points in each patch.
\end{assumption}
Under the above assumption, we provide a necessary condition for rigidity. Before doing so, we note that a \textit{lateration} criteria was earlier proposed in \cite{CKS2015a} that can guarantee rigidity. However, it is not known if there exists an efficient test for lateration. Moreover, a path configuration can be rigid without being laterated, that is, lateration is not necessary  for rigidity. This fact is demonstrated with an example in Figure \ref{PC2}. This motivated us to look for a criteria that is both necessary and sufficient for rigidity. We propose a necessary condition for rigidity that can be tested efficiently. We present some examples where the condition is also sufficient, and conjecture that this is true in general.

Before stating the result, we set up a special bipartite graph that captures the overlap-pattern among patches. Recall that a graph is said to be bipartite if the vertex set $\cV$ can be divided into disjoint subsets $\cV_1$ and $\cV_2$ such that there are no edges between the vertices of a given $\cV_i$. 
\begin{definition}[Correspondence graph]
\label{defT}
We define the bipartite correspondence graph to be $\Gamma=(\cV_1,\cV_2,\cE)$, where $\cV_1( \Gamma)$ are the nodes, $\cV_2( \Gamma)$ are the patches, and $(k,i) \in \cE( \Gamma)$ if and only if $k \in \cC_i$. 
\end{definition}
The correspondence graph for the configuration in Figure \ref{PC1} is shown on the right. Finally, we introduce a special notion of connectivity. Recall that a \textit{path} is an ordered sequence of vertices $v_1,v_2,\ldots,v_n$ such that $(v_t,v_{t+1})$ is an edge for $1 \leq t \leq n-1$. The path is said to be between vertices $\alpha$ and $\beta$ (or the path connects $\alpha$ and $\beta$) if $v_1=\alpha$ and $v_n=\beta$. Two paths in a graph are said to be \textit{vertex-disjoint} over a set $\Theta$ if they do not share a common vertex from $\Theta$. A set of paths are said to be $\Theta$-\textit{disjoint} if any two paths are vertex-disjoint over $\Theta$.
\begin{definition}[Quasi connected]
\label{defConnected}
The correspondence graph $\Gamma$ is said to be quasi $k$-connected if any two vertices in $\cV_2 \left( \Gamma \right)$ have $k$ or more $\cV_1(\Gamma)$-disjoint paths 
between them. Moreover, there exist two vertices in $\cV_2 \left( \Gamma \right)$ that are connected by exactly $k$ paths that are $\cV_1(\Gamma)$-disjoint.
\end{definition}
For latter reference, we record the following characterization of quasi $k$-connectivity. The equivalence can be derived by adapting the proof of Menger's theorem \cite[Theorem 3.3.1]{Diestel2005}. 
\begin{proposition}[]
\label{Menger}
The following are equivalent.\\
(a) The correspondence graph $\Gamma$ is quasi $k$-connected.\\
(b) $\cE(\Gamma)$ can be divided into two disjoint subsets $E_1$ and $E_2$ such that the edges from $E_1$ and that from $E_2$ are\\
\indent (i)  incident on at least $k$ common vertices from $\cV_1(\Gamma)$, and\\
\indent (ii) not incident on any common vertex from $\cV_2(\Gamma)$.
\end{proposition}
In Figure \ref{PC3}, notice that there are $3$ paths between any pair of vertices in $\cV_2(\Gamma)$, but at most two paths are $\cV_1(\Gamma)$-disjoint. In this case, the configuration is not rigid. In fact, we have the following result (cf. supplementary material for the proof). 
\begin{theorem}[Necessary condition]
\label{THEOREM}
Under Assumption \ref{assumption}, if a configuration is rigid in $\mathbb{R}^d$, then its correspondence graph must be quasi $(d\!+\!1)$-connected. 
\end{theorem}
Moreover, we see from the example in Figure \ref{PC3} that, if $\Gamma$ fails to be quasi $( d +  1)$-connected, then the configuration is not rigid. We are yet to find a counter-example where the configuration is flexible yet $\Gamma$ is quasi $( d +  1)$-connected. Based on empirical evidences, we make the following conjecture.
\begin{conjecture}
Suppose Assumption \ref{assumption} holds, and that any $k \geq d +1$ points in $\cX_s \cup \cX_a$ are non-degenerate. Then a configuration is rigid in $\mathbb{R}^d$ if and only if its correspondence graph is quasi $\left( d \! + \! 1 \right)$-connected. 
\end{conjecture}
The second assumption appears somewhat stringent at first sight. The relevance of this assumption is somewhat clear from the example in Figure \ref{PC1}. Namely, if sensors $1,2$ and anchor $5$ are concurrent, then one can reflect patch $B$ about the line joining these points. We note that the use of some form of non-degeneracy assumption is standard in rigidity theory \cite{GHT2010}. 

Based on Definitions \ref{defT} and \ref{defConnected}, it is not difficult to establish a relation between quasi connectivity and the maximum flow between the vertices of $\cV_2(\Gamma)$ \cite{CLRS2001}. In this context, recall that a vertex is said to have \textit{capacity} $\kappa$ if the incoming and outgoing flows for the vertex are at most $\kappa$ \cite{CLRS2001}. 
\begin{proposition}[Connectivity using flow]
\label{connect_cut}
Assume that each vertex in $\cV_1(\Gamma)$ is assigned unit capacity while computing the flow. Then the maximum flow between the vertices of $\cV_2(\Gamma)$ is at least $k$ if and only if $\Gamma$ is quasi $k$-connected. 
\end{proposition}
The key point is that one can efficiently check if, under the assumption that the vertices in $\cV_1(\Gamma)$ have unit capacity, the maximum flow between the vertices of $\cV_2(\Gamma)$ is at least $k$. This can be done using the Ford-Fulkerson algorithm \cite{CLRS2001}. Note that we do not need to check the maximum flow for all pairs of vertices in $\cV_2(\Gamma)$. We can simply fix a vertex and check the maximum flow between this vertex and the remaining vertices.

\begin{figure}
\center
\includegraphics[width= 0.7\linewidth]{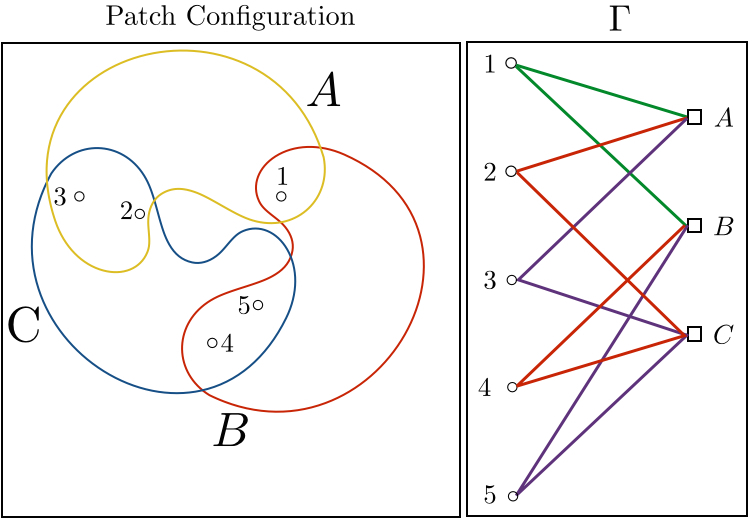}
\caption{Example of a rigid configuration (left) in two dimensions \cite{CKS2015a}. The configuration is not laterated, but the correspondence graph (right) is quasi $3$-connected. The edges of the three disjoint paths between patch vertices $A$ and $B$ are marked in solid.}
\label{PC2}
\end{figure}
\section{Partitioning}
\label{GP}

We now describe a heuristic for partitioning $\cG$ into overlapping patches such that the corresponding  $\Gamma$ is quasi $(d \! + \! 1)$-connected. In this relation, we note that divide-and-conquer approaches have been proposed in \cite{CKS2015b,CLS2012,LT2009,ZLGG2010}, where $\cG$ is partitioned into overlapping patches. The difficulty with the approaches in \cite{CKS2015b, LT2009} is that the patches and the resulting patch configuration are not guaranteed to be rigid.

We propose to bypass the former rigidity issue by using \textit{cliques}, that is, complete subgraphs of the measurement graph. In other words, each patch is a clique in our approach. This is precisely why we choose to denote the patches as $\cC_i$ in Section \ref{RP}. Cliques are trivially rigid \cite{Laman1970,GHT2010}, and can be localized using multidimensional scaling \cite{BG2003}. In particular, for each $i \in \cV(\cG)$, we  extract a maximal clique containing $i$. The system of cliques forms a clique-cover. We recall that a clique is said to be maximal if it is not contained in a strictly larger clique. By targeting maximal cliques, we wish to minimize the number of cliques that are required to be registered in the final phase.

Let $\cG_i$ denote the \textit{neighborhood} graph of some $i \in \cV(\cG)$. Namely, $\cG_i$ is the subgraph of $\cG$ induced by $i$ and its one-hop neighbors. For each vertex $i$, we want to find a maximal clique $\cC \subset \cV(\cG)$ containing $i$. In this regard, we note that it suffices to restrict the search to  $\cG_i$. 
\begin{proposition}
\label{clq1}
Let $\cC$ be a maximal clique. Then $i \in \cC$ if and only if $\cC \subset \cV(\cG_i)$. 
\end{proposition}
\begin{proof}
Let $\cC$ be a maximal clique containing $i$. Then, for any $j \in \cC$, we have $(i,j) \in \cE(\cG)$. Hence, $j \in \cV(\cG_i)$. In the other direction, suppose that $\cC \subset \cV(\cG_i)$, but $i$ does not belong to $\cC$. Then, by appending vertex $i$ to $\cC$, we obtain a clique that strictly contains $\cC$, which contradicts the maximality of $\cC$.
\end{proof}
Unfortunately, finding cliques is generally intractable \cite{BG2015}. Based on Proposition \ref{clq1}, we first extract a clique from a given subgraph $\cG_i$ using the algorithm in \cite{BG2015}. In this work, the combinatorial problem of finding maximal cliques is relaxed into a continuous optimization problem. The stationary points of the latter are computed using projected gradient descent. The key result of the paper is that one can provably locate a clique by running the  gradient-descent  for sufficient number of iterations and rounding the output \cite[Theorem 7, Corollary 3]{BG2015}. The authors empirically noticed that the clique retuned by the algorithm is often maximal. Since the subgraphs $\cG_i$ are typically small  for practical values of  $r$, we found the algorithm in \cite{BG2015} to be quite efficient for our purpose.

The clique located within a given $\cG_i$ using the above clique-finding algorithm may not contain vertex $i$. In this case, we can in fact obtain a larger clique simply by appending vertex $i$ to the found clique. Generally, since the subgraphs are small, one can efficiently test for maximality, and keep appending nodes until the maximal clique is found. In practice, we noticed that the cliques returned by the algorithm in \cite{BG2015} are often maximal or near-maximal. As a result, the combined process of appending vertices and testing for maximality is quite fast. We note that an extracted clique can belong to two or more subgraphs. We discard the redundant cliques during the clique-finding process. At the end, suppose that we have located, say, $m$ maximal cliques, $\cC_1, \dots ,$ $\cC_m$, that cover the vertices of $\cG$. We next test the rigidity of the patch configuration. To do so, we append to the existing cliques an additional clique $\cC_{m+1}$ containing the anchors, and test if $\Gamma$ is quasi $( d+1)$-connected. If so, we set $M=m$, and proceed to the localization phase in Section \ref{cMDS}.

If $\Gamma$ fails the test, we proceed as follows. Following Proposition \ref{connect_cut}, we know that there exist $s,t \in \cV_2(\Gamma)$ for which the maximum $s\text{-}t$ flow is $k \leq d$, where recall that the vertices in $\cV_1(\Gamma)$ are assigned unit capacity. In fact, the Ford-Fulkerson algorithm returns the value $k$ and the corresponding minimum cut $C=(S,T)$, where $s \in S$ and $t \in T$ \cite{CLRS2001}. Let $I_S$ and $I_T$ be the indices of the cliques in $S$ and $T$, that is, $I_S = S \cap \cV_2(\Gamma)$ and $I_T = T \cap \cV_2(\Gamma)$. 
\begin{proposition}
\label{cut1}
Let $A = \cup_{\alpha \in I_S} \cC_{\alpha}$ and $B = \cup_{\beta \in I_T} \cC_{\beta}$. Then $\lvert A \cap B \rvert = k$.
\end{proposition}
\begin{proof}
From the max-flow-min-cut theorem \cite{CLRS2001}, $\lvert A \cap B \rvert \leq k$. If $\lvert A \cap B \rvert$ is less than $k$, then there would be a vertex in $\cV_1(\Gamma)$ with more than one edge in the cut set \cite{CLRS2001}. However, since each vertex in $\cV_1(\Gamma)$ has unit capacity, this is not possible. 
\end{proof}
We wish to increase the maximum flow by extracting a clique and appending it to the existing configuration. In particular, the appended clique must contain some $i \in A\setminus B$ and $j \in B \setminus A$. Our task is to find a maximal clique containing $i$ and $j$. Define $\cG_{ij}$ be the common subgraph of $\cG_i$ and $\cG_j$, that is, $\cG_{ij}$ is the subgraph of $\cG$ induced by the vertices $\cV(\cG_i) \cap \cV(\cG_j)$. Similar to Proposition \ref{clq1}, we note the following. 
\begin{proposition} 
Let $\cC_0$ be a maximal clique. Then $i,j \in \cC_0$ if and only if $\cC_0 \subset \cV(\cG_{ij})$.
\end{proposition}
After appending $\cC_0$ to the existing cliques, we obtain a new configuration. Accordingly, we update $\cV_2(\Gamma)$ and $\cE(\Gamma)$, and increase $m$ by $1$. In particular, we reorder the indices of the cliques so that $\cC_{m+1}$ continues to be the anchor clique. For the updated $\Gamma$, we recompute $I_S$ or $I_T$, and note that $m$ belongs to either of these. As a result, we conclude the following.
\begin{proposition}
\label{NoNewMinCut}
For the updated $\Gamma$, $\lvert A \cap B \rvert > k$.
\end{proposition}
Moreover, if the maximum flow is uniquely achieved for the vertices identified by the Ford-Fulkerson algorithm, then appending $\cC_0$ actually increases the maximum flow. We continue this process, in which we alternately augment the configuration and test rigidity, until we attain the maximum flow of $d+1$. It is possible that the process is prematurely terminated if we are unable to find a clique of size at least $d +1$. This typically happens in adversarial settings where $r$ is small, making $\cG$ extremely sparse. At the end of the process, assume that we have $M+1$ cliques, $\cC_1\ldots,\cC_{M+1}$, where $\cC_{M+1}$ is the anchor clique. For the simulations in Section \ref{exp}, we found that $M$ is typically about $30\%$ of the total number of nodes.

\section{Localization}
\label{cMDS}

Having partitioned the measurement graph into a system of overlapping cliques, we now localize them in parallel. Since all the inter-node distances are available in a clique, we can efficiently localize a clique using multidimensional scaling \cite{BG2003}. However, there are two types of cliques, namely, cliques without anchors and those with anchors. For the former, we can directly use classical multidimensional scaling (cMDS). In particular, suppose that the clique has $n$ sensors, and the distances are $\{d_{ij} : 1 \leq i,j \leq n\}$. Consider the $n \times n$ matrices $\bD$ and $\bB$ given by
\[ \bD_{ij} = \begin{cases} 
d_{ij}^2 & \text{ if } \ i \neq j, \\
0 & \text{ otherwise,} \\
\end{cases}
\]
and 
\begin{equation*}
\bB = -\frac{1}{2} \Big(\bI_n-\frac{1}{n}\textbf{uu}^{\top}\Big) \mathbf{D} \Big(\bI_n-\frac{1}{n}\textbf{uu}^{\top} \Big),
\end{equation*}
where $\textbf{u}$ is the all-ones vector of length $n$, and $\bI_n$ is the $n \times n$ identity matrix. Since $\bB$ is symmetric, it has real eigenvalues and a full set of orthonormal eigenvectors. Let $\lambda_1 \geq  \dots \geq \lambda_n$ be the sorted eigenvalues, and $\q_1,\ldots,\q_n$ be the corresponding eigenvectors. 
\begin{theorem}[Multidimensional Scaling, \cite{BG2003}]
Suppose that the available distances are exact, that is, $d_{ij}=\lVert \x_i - \x_j \rVert$ for some $\x_1,\ldots,\x_n$. Then $\bB \succeq \mathbf{0}$ and $\mathrm{rank}(\bB) \leq d$. The sensor locations can be taken to be 
\begin{equation}
\label{embedding}
\x_i = \big( \sqrt{\lambda_1} \q_1(i), \dots, \sqrt{\lambda_d} \q_d(i) \big)^{\top} \quad (1\leq i \leq n).
\end{equation}
\end{theorem}
If the distances are noisy, $\bB$ can have negative eigenvalues and its rank can be greater than $d$. In this case, it is customary to use the positive eigenvalues and the corresponding eigenvectors in \eqref{embedding}. The resulting inter-sensor distances are an approximation to the available distances, where the approximation error is determined by the rank of $\bB$, and the number and magnitude of the negative eigenvalues \cite{BG2003}. Perturbation analysis of cMDS is a well-researched topic and the method is known to be stable under deformations \cite{Sibson1979}.

If a clique has one or more anchors, we have to take into consideration the stipulated anchor locations. More precisely, we have a constrained problem, where we need to reconstruct the sensor locations keeping the anchor variables fixed. In such scenarios, we can use cMDS followed by an alignment. Assume that, of the $n$ nodes, the first $k$ are anchors and the remaining are sensors. We first localize the $n$ nodes using cMDS, regardless of the anchor locations. Then we align the reconstructed anchors with the original anchors via a rigid transformation. In particular, if there is just one anchor, and if the reconstructed and original locations are $\x$ and $\bar{\a}$, then we translate the reconstructed nodes by $\bar{\a}-\x$. If there are more than one anchor, then we perform an optimal alignment using least-square fitting:
\begin{equation}
\label{Arun}
\min_{\bO \in \O\left(d\right), \t \in \R^d} \ \sum_{i=1}^{k} \lVert \bO \x_{i}+\t  - \bar{\a}_i \rVert^2,
\end{equation}
where $\x_i$ and $\bar{\a}_i$ are the reconstructed and the original anchor locations. As is well-known, the minimum of \eqref{Arun} has a simple closed-form solution \cite{AHB1987}. In particular, the optimal transform is given by $\bO^{\star}=\bV \bU^{\top}$ and $\t^{\star} = \bm{\mu} -\bO^{\star}\bm{\nu}$, where 
\begin{equation*}
\bm{\mu}= \frac{1}{k} \sum_{i=1}^{k} \x_i \quad \text{and} \quad\bm{\nu} = \frac{1}{k} \sum_{i=1}^k \bar{\a}_i,
\end{equation*}
and $\bC=\bU \Sigma \bV^{\top}$ is the SVD of $\bC = \sum_{i=1}^{k} (\x_i - \bm{\mu})(\bar{\a}_i -\bm{\nu})^{\top}$. We apply the transform $\x \mapsto \bO^{\star}\x+\t^{\star}$ on the reconstructed sensors, and place the anchors in their stipulated locations. Finally, we refine the localization by minimizing the stress function using a gradient-based method \cite{BLTYW2006}. The refinement is particularly effective when the distances are noisy.

\section{Registration}
\label{ADMM}

As a final step, we need to register the localized cliques in a global coordinate system. While the least-square formulation of the registration problem has a closed-form solution for two cliques \cite{AHB1987}, the problem is generally intractable when there are three or more cliques \cite{CKS2015a}. Recently, it was demonstrated in \cite{CKS2015a} that the least-square optimization can be approximated using a semidefinite program (SDP). Later, a scalable ADMM-based solver for this SDP was proposed in \cite{SC2016}.
For completeness, we review the SDP relaxation and the ADMM solver. 

In the absence of noise, the relation between the local and global coordinates are given by \eqref{LocGlob1} and \eqref{LocGlob2}. Since these are not expected to hold exactly in the presence of noise, the authors in \cite{CKS2015b} proposed to minimize the least-square objective
\begin{equation}
\label{LS}
\sum_{i=1}^{M} \Big(\sum_{k \in \cC_i \setminus \cA} \alpha_{k,i}^2  + \lambda \sum_{l \in \cC_i \cap \cA}   \beta_{l,i}^2 \Big),
\end{equation}
where 
\begin{equation*}
\alpha_{k,i}= \lVert \x_k - \bO_i \x_{k,i} - \t_i\rVert   \ \ \text{and} \ \ \beta_{l,i}=\lVert \bO_{M+1}\bar{\a}_l - \bO_i \bar{\a}_l - \t_i \rVert
\end{equation*}
are the registration errors for sensors and anchors. The scale $\lambda > 0$ is used to combine the gross errors. The variables are the sensor coordinates $\x_k$ and the rigid transformations $(\bO_i, \t_i)$. The dummy variable $\bO_{M+1}$ is introduced to make the objective homogenous \cite{CKS2015b}. In terms of the matrix variables
\begin{equation*}
\bZ = \left[ \x_1 \cdots \x_{N} \ \t_1 \dots \t_{M} \right] \quad \text{and} \quad \bO = \left[ \bO_1 \cdots \bO_{M+1} \right],
\end{equation*}
we can express \eqref{LS} as
\begin{equation}
\label{LS2}
\text{Trace} \left( \begin{bmatrix} \bZ & \bO \end{bmatrix} \begin{bmatrix} \bJ & -\bB^{\top} \\ -\bB & \bD  \end{bmatrix}  \begin{bmatrix} \bZ^{\top} \\ \bO^{\top} \end{bmatrix} \right),
\end{equation}
where
\begin{equation*}
\begin{aligned}
& \bJ = \sum_{i = 1}^{M} \Big[ \sum_{k \in \cC_i \setminus \cA}  \!\! \e_{k,i} \e_{k,i}^{\top} + \lambda \sum_{l \in \cC_i \cap \cA} \!\! \del_{N + i}^{N + M} {\del_{N + i}^{N + M}}^{\top} \Big], \\
& \bB = \sum_{i = 1}^{M+1} \Big[ \sum_{k \in \cC_i \setminus \cA} \!\! \left( \del_{i}^{M+1} \otimes \bI_d \right)\x_{k,i} \e_{k,i}^{\top} \\
& \quad + \lambda \sum_{l \in \cC_i \cap \cA} \!\! \left( \f_{i} \otimes \bI_d \right)\bar{\a}_l {\del_{N+i}^{N+M}}^{\top} \Big],\\
& \bD = \sum_{i = 1}^{M+1} \Big[ \sum_{k \in \cC_i \setminus \cA}  \!\!\left( \del_{i}^{M+1} \otimes \bI_d \right)\x_{k,i} \x_{k,i}^{\top} \left( \del_{i}^{M+1} \otimes \bI_d \right)^{\top}\\
& \quad + \lambda \sum_{l \in \cC_i \cap \cA} \!\!\left( \f_{i} \otimes \bI_d \right)\bar{\a}_l \bar{\a}_l^{\top} \left( \f_{i} \otimes \bI_d \right)^{\top} \Big].
\end{aligned}
\end{equation*}
Here, $\otimes$ is the Kronecker product, $\del_i^{L}$ is the all-zero vector of length $L$ with unity at the $i$-th position, 
\begin{equation*}
\e_{k,i} =  \del_{k}^{N+M}-\del_{N+i}^{N+M} \quad \text{and} \quad \f_i =\del_{M+1}^{M+1}-\del_{i}^{M+1}.
\end{equation*}
The minimum of \eqref{LS2} over $\bZ$ is attained when $\bZ^{\star} = \bO \bB \bJ^{-1}$. On substituting $\bZ^{\star}$ in \eqref{LS2}, we get the following problem:
\begin{equation}
\label{GRET}
\begin{aligned}
& \underset{\bG \succeq 0}{\text{min}} \quad \text{Trace} \left( \bC \bG \right) \\
& \text{s.t.} \quad \left[\bG\right]_{ii} =\bI_d \  (i = 1, \ldots, M\!+\!1), \ \text{rank}\left( \bG \right)=d.
\end{aligned}
\end{equation}
where $\bC = \bD - \bB \bJ^{-1} \bB^{\top}, \bG = \bO^{\top} \bO,$ and the $d \times d$ matrix $\left[\bG\right]_{ii}$ denotes the $i$-th diagonal block of $\bG$. It was observed in \cite{CKS2015a} that the objective and the constraints in \eqref{GRET} are convex, except for the rank condition. By dropping the rank constraint, the authors arrived at the following SDP relaxation:
\begin{equation}
\label{GRET-SDP}
 \min_{\bG \succeq 0} \ \mathrm{Trace}\left(\bC \bG \right) \ \text{s.t.} \ \left[\bG \right]_{ii}=\bI_d \ \ (i = 1, \ldots, M\!+\!1). 
\end{equation}
The global minimum of  \eqref{GRET-SDP} can be computed for small or even medium-sized problems using an interior-point solver \cite{Sturm1999}. However, such solvers are both memory and computation intensive. In particular, the cost of approximating the global minimum of \eqref{GRET-SDP} within a given accuracy is $\cO(\left(M d\right)^{4.5})$ \cite{CKS2015a}. Since $M$ is of the order $\cO(|\cN|)$ in our case, the size of the SDP variables can be few hundreds or thousands. Interior-point solvers run out of memory for such large problems. To achieve scalability, an iterative solver based on ADMM was proposed in \cite{SC2016}. The ADMM updates are summarized in Algorithm \ref{solver}, where 
$\S_+$ is the set of symmetric positive semidefinite matrices of size $L=(M+1)d$, $\Omega$ is the set of symmetric matrices of size $L$ whose $d \times d$ diagonal blocks are identity, and 
$\Pi_\cS (\bA)$ denotes the projection of $\bA$ onto a convex set $\cS$. Notice that the only non-trivial computation is determining $\Pi_{\S_+} (\bA)$. This is obtained by computing the eigendecomposition of $\bA$ and setting the negative eigenvalues to zero. The projection $\Pi_\Omega (\bA)$ amounts to setting the $d \times d$ diagonal blocks of $\bA$ to $\bI_d$, while keeping the non-diagonal blocks unchanged. We initialize $\bH$ using the spectral algorithm in \cite{CKS2015a}. The Lagrange multiplier $\Lam$ is initially set to be the zero matrix. We use a condition from \cite{BPCPE2011} to terminate the iterations. 

\begin{algorithm}
\label{solver}
\DontPrintSemicolon
\KwIn{$\bC, \rho > 0$}
\KwOut{$\bG$.}
Initialize $\bH$ and $\Lam.$\;
\While{stopping criteria is not met}
{
$\bG \longleftarrow \Pi_{\S_+} \left[\bH - \rho^{-1} \left(\bC-\Lam \right) \right]$.\; \label{Gupdate}
$\bH \longleftarrow \Pi_{\Omega} \left[\bG - \rho^{-1} \Lam \right]$.\; \label{Hupdate}
$\Lam \longleftarrow \Lam + \rho(\bH-\bG)$.\;
}
\caption{ADMM Solver.}
\end{algorithm}

We can establish the convergence of Algorithm \ref{solver} using the analysis in \cite{BPCPE2011}. In particular, we have the following result.
\begin{theorem}
\label{thmADMM}
Starting with $\bH^0$ and $\Lam^0$, let  $(\bG^{k},\bH^{k},\Lam^{k})_{k \geq 1}$ be the variables generated by Algorithm \ref{solver}. Then
\begin{itemize}
\item Objective convergence: If $F^*$ is the optimum of \eqref{GRET-SDP}, then
\begin{equation*}
\label{conv1}
\lim_{k \rightarrow \infty} \ \mathrm{Trace}(\bC \bG^k) = F^*.
\end{equation*}
\item Asymptotic feasibility: For $1 \leq i \leq M$,
\begin{equation*}
\label{conv2}
\lim_{k \rightarrow \infty} \ [\bG^k]_{ii} = \bI_d.
\end{equation*}
That is, $(\bG^k)$ approaches the feasible set in \eqref{GRET-SDP}.
\end{itemize}
\end{theorem}
In fact, since updates \ref{Gupdate} and \ref{Hupdate} in Algorithm \ref{solver} are convex projections, we can establish the Theorem \ref{thmADMM} using elementary results from convex analysis. This and other technical results will be reported separately \cite{JSC2016}.

Having approximated the optimal $\bG$ using Algorithm \ref{solver}, we compute  $\bO =[ \bO_1 \cdots  \bO_{M+1}]$ using the rounding in \cite{CKS2015a}. The first $N$ columns of $\bZ = \bO \bB \bJ^{-1}$ are taken to be the estimated sensor locations. As a final step, we refine the locations using stress minimization \cite{BLTYW2006} and denote the result as $\widehat{\x}_1, \dots , \widehat{\x}_N$.

\begin{figure}
	\centering
	\subfloat[$\text{ANE}= 1.9\mbox{e-}1$.]{\includegraphics[width= 0.48 \linewidth]{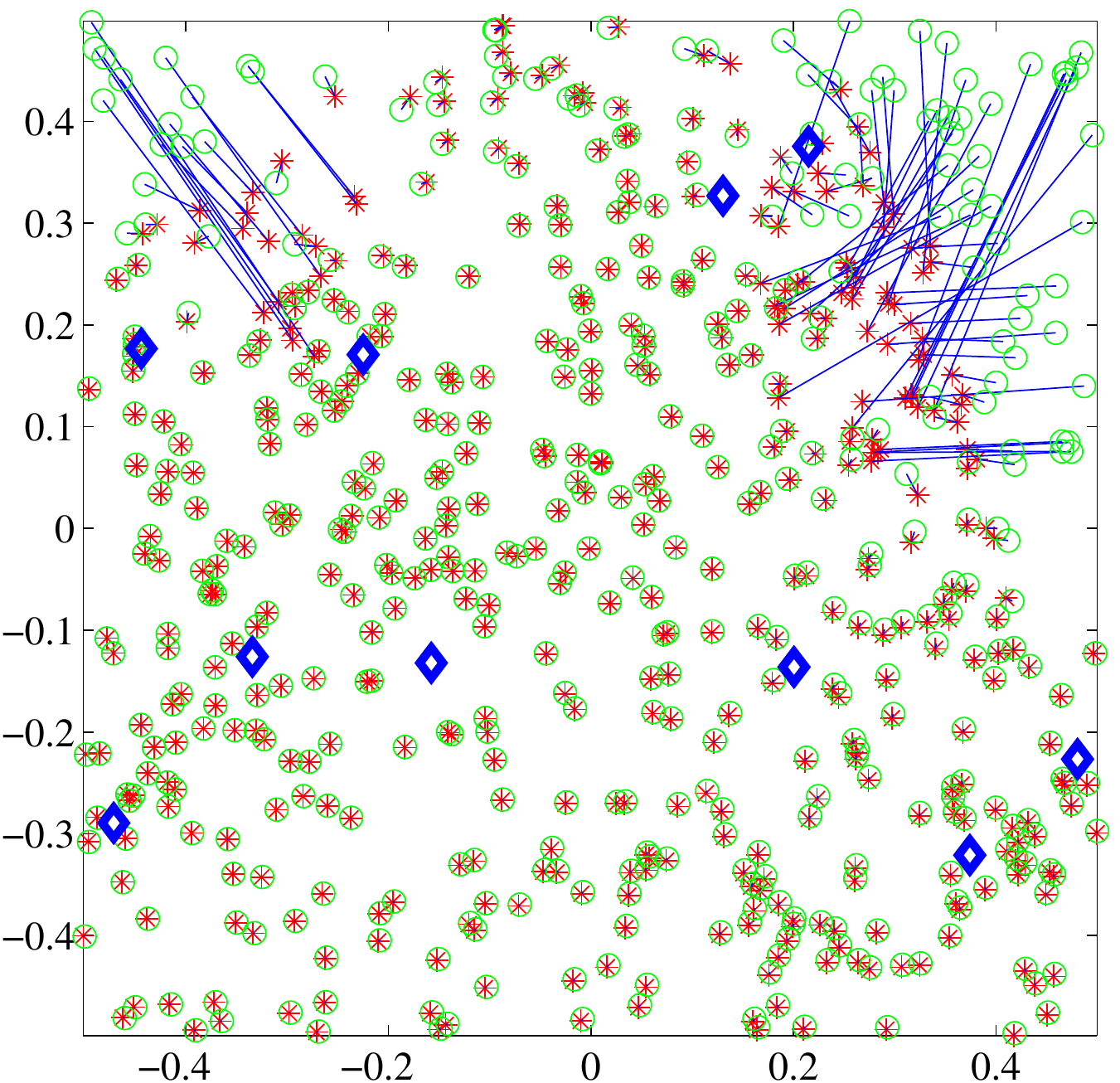}} \hspace{-0.5mm}
	\subfloat[$\text{ANE}=7.1\mbox{e-}12$.]{\includegraphics[width= 0.48 \linewidth]{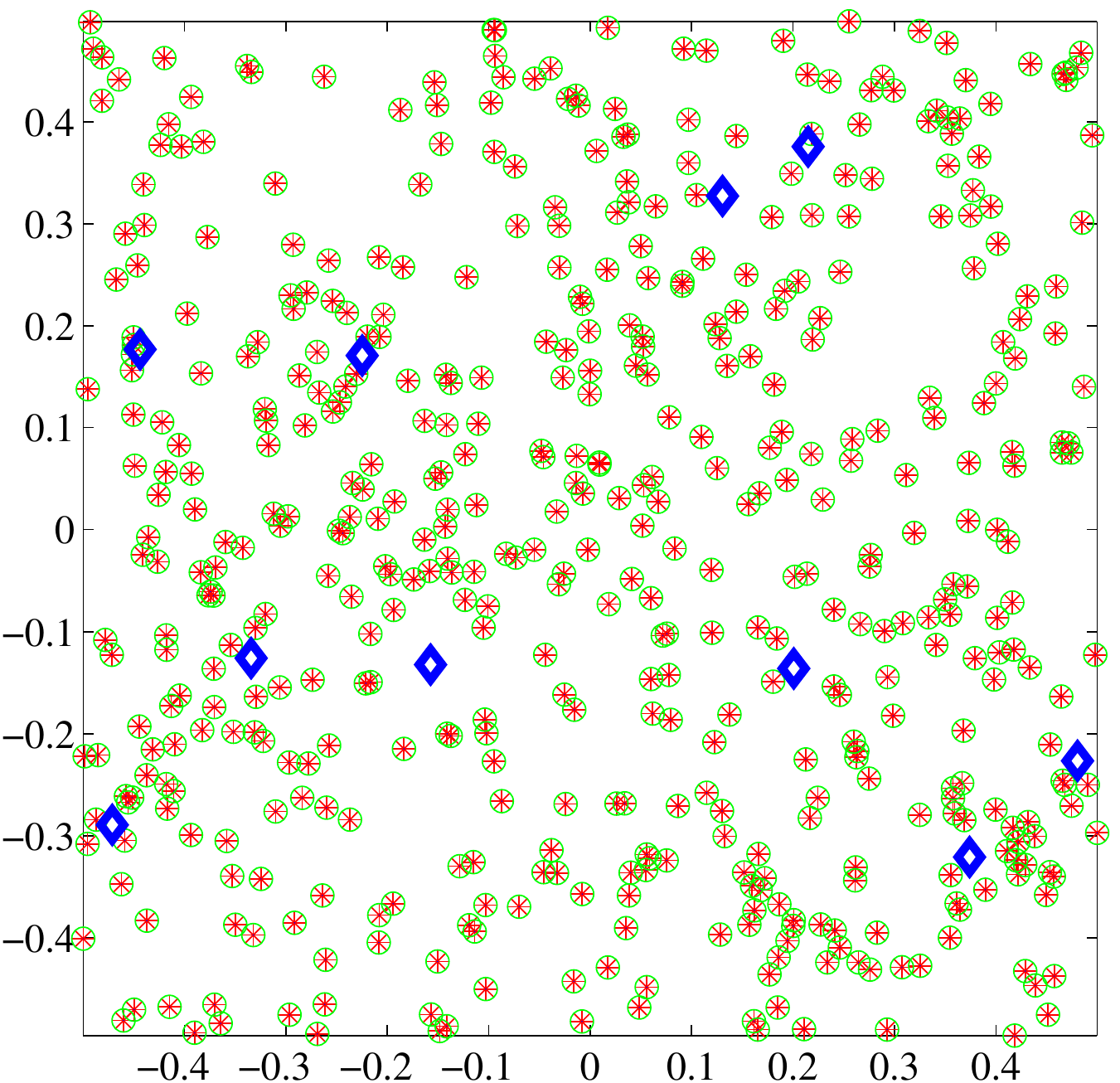}} \vspace{0.2mm}
	\subfloat[$\text{ANE}=8.6\mbox{e-}2$.]{\includegraphics[width= 0.48 \linewidth]{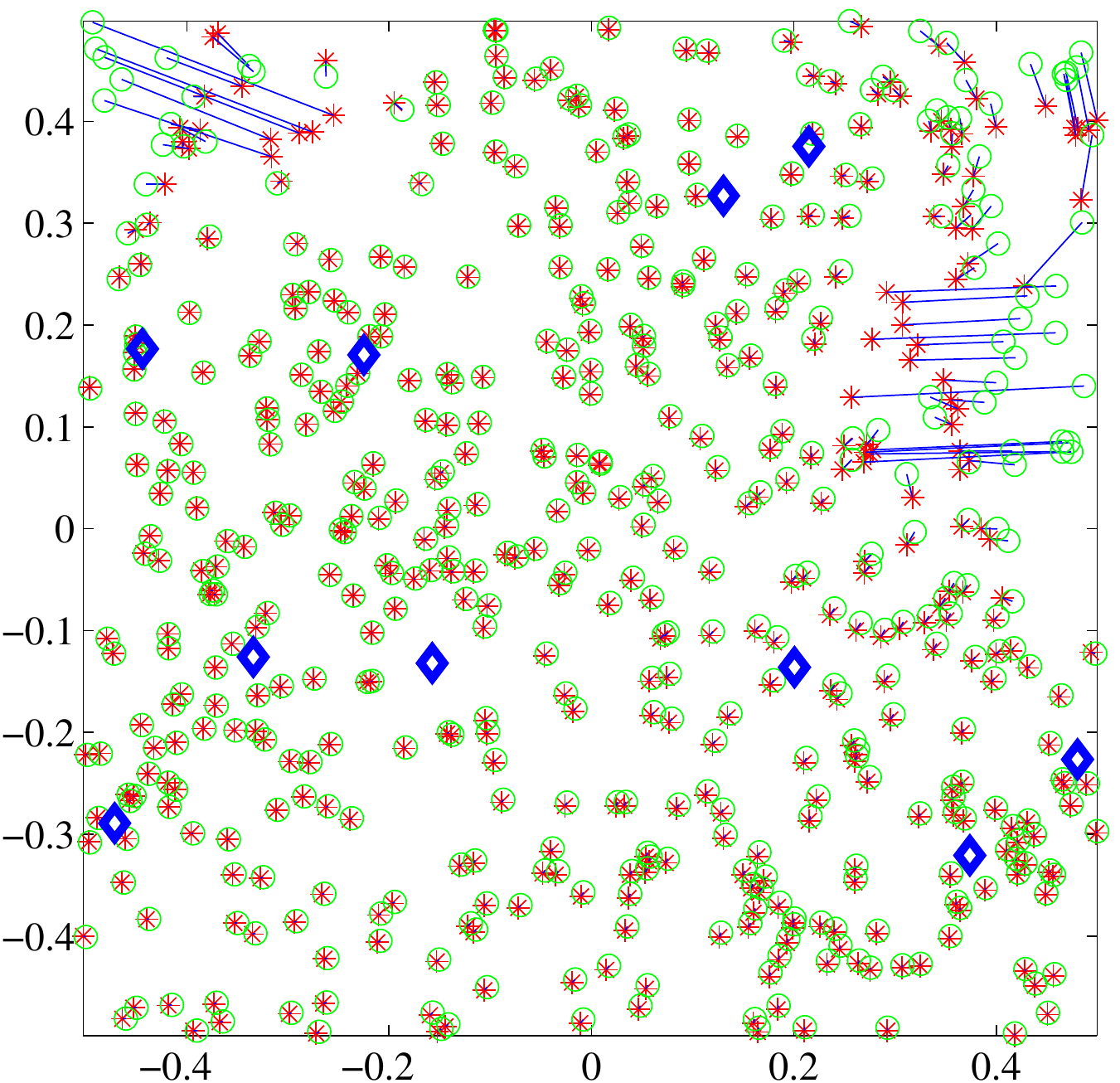}} \hspace{-0.5mm}
	\subfloat[$\text{ANE}=6.1\mbox{e-}3$.]{\includegraphics[width= 0.48 \linewidth]{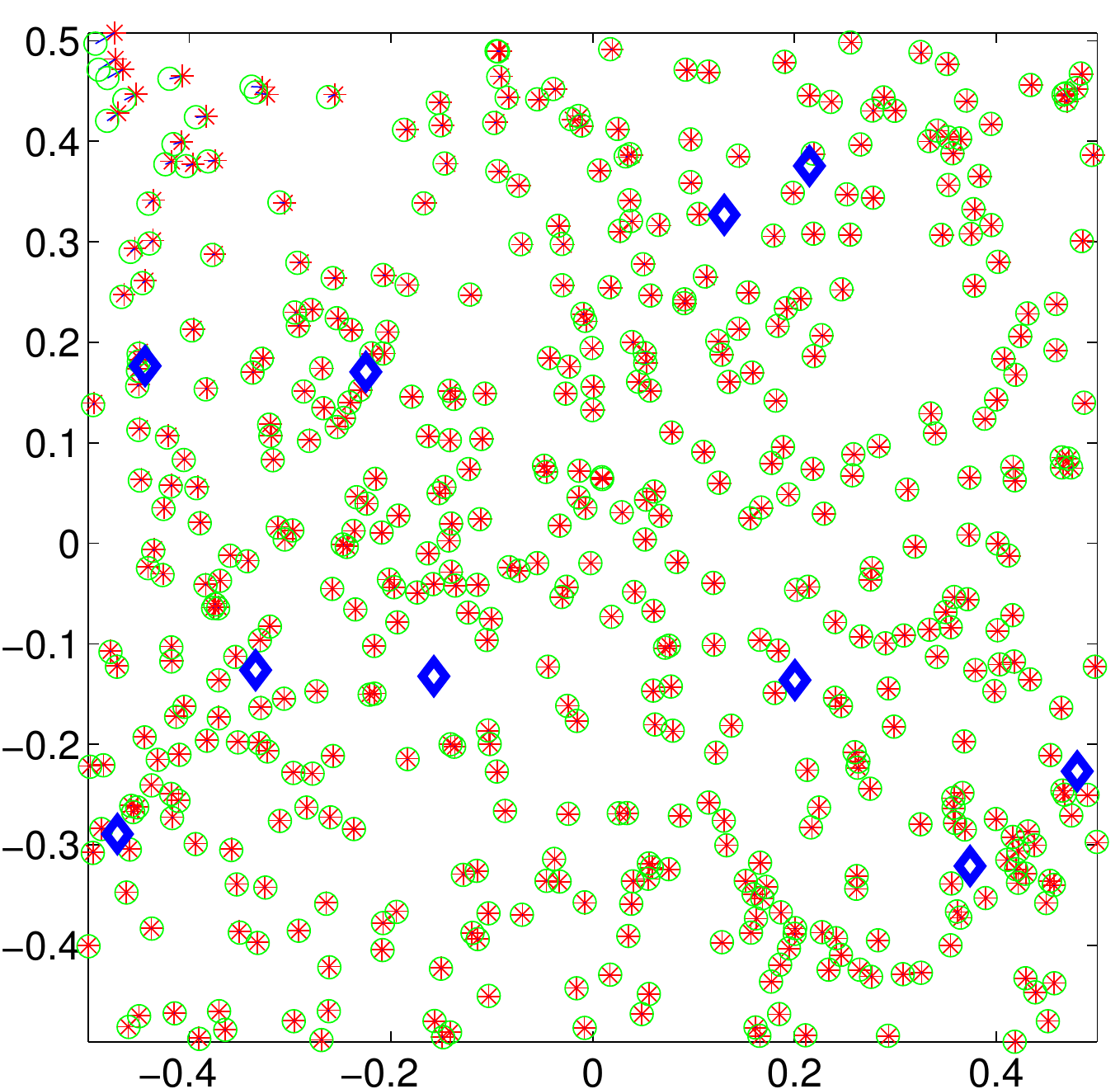}}
	\caption{Illustration of the impact of rigidity on the registration accuracy. The parameters for the RGG are $N = 500, K =10,$ and $r=0.17$. The top and bottom rows correspond to the noise levels $\eta = 0$ and $\eta = 0.01$. For a fixed $\eta$, the figure on the left corresponds to the situation where $\Gamma$ fails to be quasi $3$-connected, while that on the right corresponds to the situation where $\Gamma$ is modified to ensure quasi $3$-connectivity (using the procedure described in the text). Green circles (\textcolor{green}{$\circ$}) denote original sensor locations, red stars (\textcolor{red}{$\star$}) denote estimated locations, and blue diamonds (\textcolor{blue}{$\Diamond$}) denote anchor locations.}
\label{GammaRigid}
\end{figure}

\section{Numerical Experiments}
\label{exp}

In this section, we conduct numerical simulations to demonstrate the performance of the proposed method. In particular, we illustrate the impact of rigidity on the performance of the registration algorithm, and study the timing of different phases of the proposed method and the scaling of the localization error with the noise level. We then compare with some of the state-of-the-art algorithms \cite{SL2014,BLTYW2006,WZYB2008,SXG2015} in terms of accuracy, run-time, and scalability. The comparisons are performed on planar networks, namely,  the random geometric graph (RGG) \cite{BLTYW2006,WZYB2008,CKS2015b}, and the structured PACM logo \cite{CLS2012}. The diameter (maximum distance between any two points) of the logo is $16.92$. We consider the following noise model that was used in \cite{BLTYW2006, WZYB2008, CKS2015b}:
\begin{equation*}
d_{ij} =\lvert 1+\epsilon_{ij} \rvert \cdot \lVert \bar{\x}_i-\bar{\x}_j \rVert \qquad (i,j \in \cS),
\end{equation*}
and
\begin{equation*}
d_{ik} = \lvert 1+\epsilon_{ik}\rvert \cdot \lVert \bar{\x}_i - \bar{\a}_k \rVert \qquad (i \in \cS, k \in \cA),
\end{equation*}
where $\epsilon_{ij}$ and $\epsilon_{ik}$ are i.i.d. Gaussians with mean zero and standard deviation $\eta$. As mentioned earlier, we enforce symmetry by replacing $d_{ij}$ and $d_{ji}$ with their average \cite{SHCJ2010}. For a quantitative comparison of the localization accuracies, we use the average normalized error (ANE) \cite{CLS2012} given by
\begin{equation*}
	\text{ANE}=\left\{\frac{\sum_{i=1}^{N} \lVert \widehat{\x}_i - \bar{\x}_i \rVert^2}{\sum_{i=1}^{N} \lVert \bar{\x}_i - \bar{\x}_c \rVert^2} \right\}^{1/2},
\end{equation*}
where $\bar{\x}_c$ is the centroid of the original sensor locations. Of course, we assume that the reconstruction has been optimally aligned with the ground truth before computing the ANE \cite{AHB1987}. We also present visual comparison of the localizations obtained using different methods. For all experiments, we have used $\lambda=1$ in \eqref{LS} and $\rho=0.01$ for the augmented Lagrangian.

\begin{table}[!htb]
	\centering
	\caption{Run-times (in seconds) of different phases of the algorithm -- partitioning $(t_1)$, localization $(t_2)$ and registration $(t_3)$.}
	\label{SelfCompare}
		\begin{tabular}{|c|c|c|c|c|c|c|}
			\hline
			$N$                 & $K$                & $r$                   & $\eta$ & $t_1$ & $t_2 $ & $t_3$ \\ \hline
			\multirow{2}{*}{100}  & \multirow{2}{*}{10}  & \multirow{2}{*}{0.4}  & 0        & 0.79   & 0.04  & 0.02 \\ \cline{4-7} 
			&                      &                       & 0.1      & 0.73   & 0.06  & 0.92   \\ \hline
			\multirow{2}{*}{500}  & \multirow{2}{*}{50}  & \multirow{2}{*}{0.18} & 0        & 3.9   & 0.07   & 0.21 \\ \cline{4-7} 
			&                      &                       & 0.1      & 3.9   & 0.1   & 12.8  \\ \hline
			\multirow{2}{*}{800}  & \multirow{2}{*}{80}  & \multirow{2}{*}{0.14} & 0        & 7.2   & 0.1   & 0.8  \\ \cline{4-7} 
			&                      &                       & 0.1      & 7.5   & 0.2   & 70.9 \\ \hline
			\multirow{2}{*}{1000} & \multirow{2}{*}{100} & \multirow{2}{*}{0.12} & 0        & 9.2   & 0.1    & 1.7   \\ \cline{4-7} 
			&                      &                       & 0.1      & 8.1  & 0.2     & 1.5 \\ \hline
		\end{tabular}
\end{table}

\begin{figure}
   \centering
	\includegraphics[width= 0.5\linewidth]{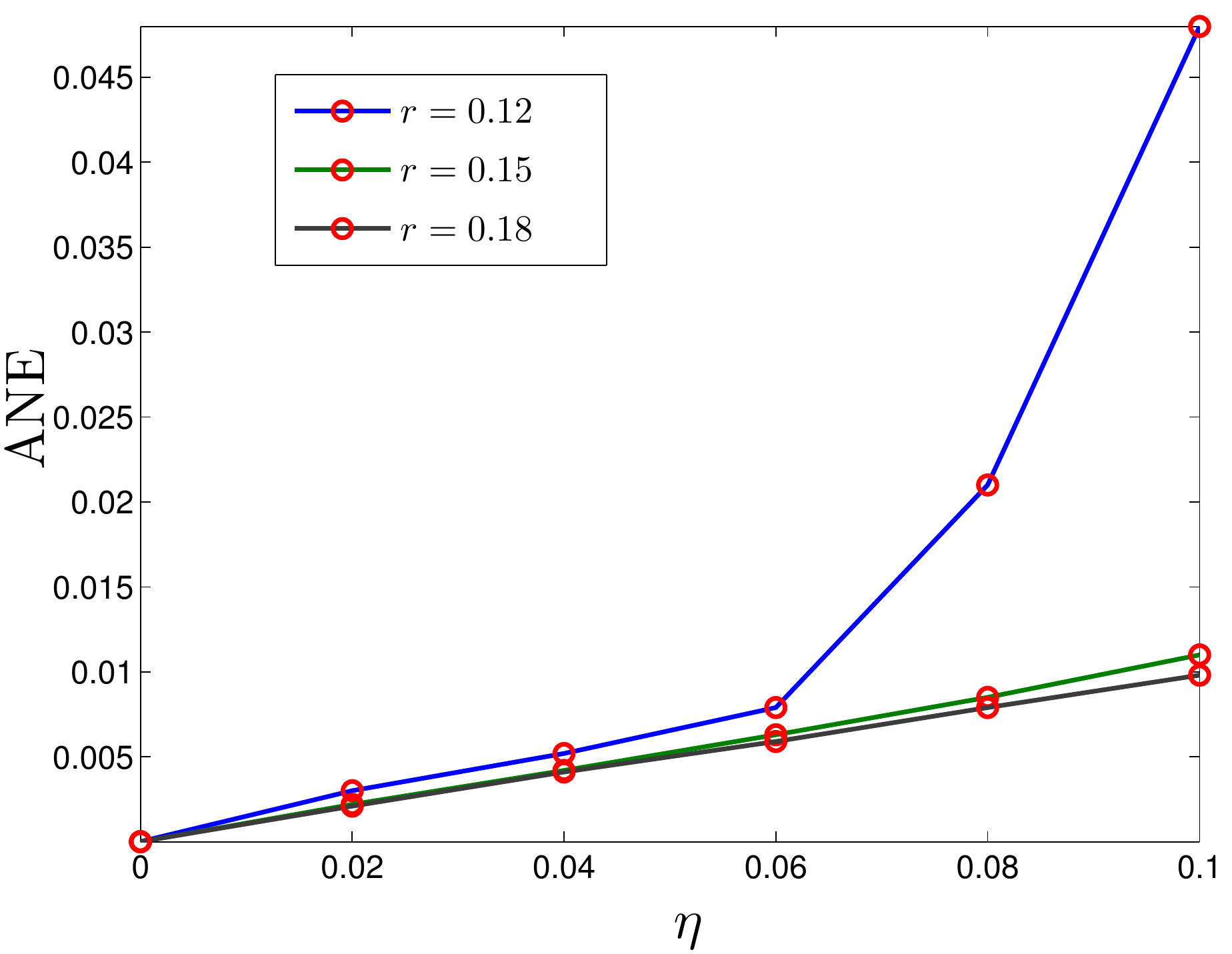}
	\caption{ANE versus noise level $\eta$ for a RGG consisting of $500$ sensors and $50$ anchors. We consider the radio ranges $r=0.12,0.15,\text{and } 0.18$.}
	\label{Self_perfor}
\end{figure}

\subsection{Performance Analysis}

\begin{table*}[!htbp]
\centering
\caption{Comparison of the run-time of the proposed method with that of \texttt{SNLSDP} \cite{BLTYW2006},  \texttt{ESDP} \cite{WZYB2008} and the localization accuracy of the proposed method with \texttt{E-ML} \cite{SL2014}, \texttt{SNLSDP} \cite{BLTYW2006},  \texttt{ESDP} \cite{WZYB2008} and \texttt{SNLDR} \cite{SXG2015}, for random geometric graphs on the unit-square. The run-time and ANE were averaged over $10$ realizations of the random graph and the measured distances. We have used `$-$' to mark those instances where the algorithm took indefinite time to solve the problem. The instances where the interior point solver ran out of memory are marked with $\star$.}
\label{Compare}

\begin{tabular}{cccc|c|c|c||c|c|c|c|c|}
\cline{5-12}   &                                             &                                              &          & \multicolumn{3}{c||}{Time}                                                                      & \multicolumn{5}{c|}{Accuracy (ANE)}                                                 \\ \hline
\multicolumn{1}{|c|}{$N$}                   & \multicolumn{1}{c|}{$K$}                  & \multicolumn{1}{c|}{$r$}                     & $\eta$ & Proposed & \texttt{ESDP} \cite{WZYB2008} & \texttt{SNLSDP} \cite{BLTYW2006} & Proposed & \texttt{ESDP} \cite{WZYB2008} & \texttt{SNLDR} \cite{SXG2015} & \texttt{SNLSDP} \cite{BLTYW2006} & \texttt{E-ML} \cite{SL2014} \\ \hline
\multicolumn{1}{|c|}{\multirow{2}{*}{$10$}}   & \multicolumn{1}{c|}{\multirow{2}{*}{$5$}}   & \multicolumn{1}{c|}{\multirow{2}{*}{$1.25$}} & $0$      & $0.1$sec           & $0.2$sec          & $0.3$sec                    & $3.9\mbox{e-}16$ & $7.9\mbox{e-}8$   & $1.3\mbox{e-}4$  & $1.3\mbox{e-}9$    & $1.4\mbox{e-}3$ \\ \cline{4-12} 
\multicolumn{1}{|c|}{}                        & \multicolumn{1}{c|}{}                       & \multicolumn{1}{c|}{}                        & $0.1$    & $0.2$sec           & $0.3$sec          & $0.4$sec                    & $9.6\mbox{e-}2$  & $9.5\mbox{e-}2$   & $2\mbox{e-}1$    & $9.6\mbox{e-}2$    & $1.3\mbox{e-}1$ \\ \hline
\multicolumn{1}{|c|}{\multirow{2}{*}{$20$}}   & \multicolumn{1}{c|}{\multirow{2}{*}{$6$}}   & \multicolumn{1}{c|}{\multirow{2}{*}{$0.88$}} & $0$      & $0.3$sec           & $0.6$sec          & $0.5$sec                    & $1.3\mbox{e-}15$ & $1.1\mbox{e-}8$   & $1\mbox{e-}4$  & $1.1\mbox{e-}8$    & $2.5\mbox{e-}3$ \\ \cline{4-12} 
\multicolumn{1}{|c|}{}                        & \multicolumn{1}{c|}{}                       & \multicolumn{1}{c|}{}                        & $0.1$    & $0.6$sec           & $0.6$sec          & $1$sec                      & $6.4\mbox{e-}2$  & $8.8\mbox{e-}2$   & $1.6\mbox{e-}1$    & $6.4\mbox{e-}2$    & $9.2\mbox{e-}2$ \\ \hline
\multicolumn{1}{|c|}{\multirow{2}{*}{$40$}}   & \multicolumn{1}{c|}{\multirow{2}{*}{$8$}}   & \multicolumn{1}{c|}{\multirow{2}{*}{$0.63$}} & $0$      & $0.5$sec           & $3.3$sec          & $0.6$sec                    & $2.3\mbox{e-}15$ & $4.1\mbox{e-}8$   & $1\mbox{e-}2$  & $1.2\mbox{e-}9$    & $2.3\mbox{e-}3$ \\ \cline{4-12} 
\multicolumn{1}{|c|}{}                        & \multicolumn{1}{c|}{}                       & \multicolumn{1}{c|}{}                        & $0.1$    & $1.2$sec           & $1.2$sec          & $0.8$sec                    & $4\mbox{e-}2$  & $7.3\mbox{e-}2$   & $1.5\mbox{e-}1$    & $4\mbox{e-}2$    & $9.6\mbox{e-}2$ \\ \hline
\multicolumn{1}{|c|}{\multirow{2}{*}{$200$}}  & \multicolumn{1}{c|}{\multirow{2}{*}{$24$}}  & \multicolumn{1}{c|}{\multirow{2}{*}{$0.28$}} & $0$      & $2$sec             & $30$sec      & $19$sec                          & $4\mbox{e-}14$   & $3.1\mbox{e-}7$   & $1.4\mbox{e-}2$  & $2.5\mbox{e-}9$    & $-$             \\ \cline{4-12} 
\multicolumn{1}{|c|}{}                        & \multicolumn{1}{c|}{}                       & \multicolumn{1}{c|}{}                        & $0.1$    & $4$sec             & $7$sec       & $4$sec                           & $1.7\mbox{e-}2$  & $3.1\mbox{e-}2$   & $1.1\mbox{e-}1$  & $1.7\mbox{e-}2$    & $-$             \\ \hline
\multicolumn{1}{|c|}{\multirow{2}{*}{$500$}}  & \multicolumn{1}{c|}{\multirow{2}{*}{$54$}}  & \multicolumn{1}{c|}{\multirow{2}{*}{$0.18$}} & $0$      & $5$sec             & $1.5$min     & $5.8$min                         & $4.7\mbox{e-}14$ & $1.5\mbox{e-}7$   & $1.6\mbox{e-}2$  & $8.8\mbox{e-}7$    & $-$             \\ \cline{4-12} 
\multicolumn{1}{|c|}{}                        & \multicolumn{1}{c|}{}                       & \multicolumn{1}{c|}{}                        & $0.1$    & $1$min       & $25$sec            & $7.6$min                         & $1\mbox{e-}2$    & $2\mbox{e-}2$      & $7.2\mbox{e-}2$  & $1\mbox{e-}2$      & $-$             \\ \hline
\multicolumn{1}{|c|}{\multirow{2}{*}{$1000$}} & \multicolumn{1}{c|}{\multirow{2}{*}{$104$}} & \multicolumn{1}{c|}{\multirow{2}{*}{$0.12$}} & $0$      & $10$sec            & $2.6$min     & $26$min                          & $1.3\mbox{e-}13$ & $9.9\mbox{e-}7$   & $5.9\mbox{e-}3$  & $3.6\mbox{e-}2$    & $-$             \\ \cline{4-12} 
\multicolumn{1}{|c|}{}                        & \multicolumn{1}{c|}{}                       & \multicolumn{1}{c|}{}                        & $0.1$    & $5.4$min      & $1.2$min          & $26$min                          & $7\mbox{e-}3$    & $1.3\mbox{e-}2$   & $4.6\mbox{e-}2$  & $4.1\mbox{e-}2$   & $-$             \\ \hline
\multicolumn{1}{|c|}{\multirow{2}{*}{$4000$}} & \multicolumn{1}{c|}{\multirow{2}{*}{$404$}} & \multicolumn{1}{c|}{\multirow{2}{*}{$0.06$}} & $0$      & $3.2$min      & $32.8$min         & $\star$                          & $5.6\mbox{e-}13$ & $2.1\mbox{e-}7$   & $-$             & $-$                & $-$             \\ \cline{4-12} 
\multicolumn{1}{|c|}{}                        & \multicolumn{1}{c|}{}                       & \multicolumn{1}{c|}{}                        & $0.05$   & $4.2$min      & $32.8$min         & $\star$                          & $1.7\mbox{e-}3$  & $3.2\mbox{e-}3$   & $-$              & $-$                & $-$             \\ \hline
\multicolumn{1}{|c|}{\multirow{2}{*}{$6000$}} & \multicolumn{1}{c|}{\multirow{2}{*}{$604$}} & \multicolumn{1}{c|}{\multirow{2}{*}{$0.05$}} & $0$      & $8.5$min      & $1$hr             & $\star$                          & $7.6\mbox{e-}13$ & $2.3\mbox{e-}3$   & $-$             & $-$                & $-$             \\ \cline{4-12} 
\multicolumn{1}{|c|}{}                        & \multicolumn{1}{c|}{}                       & \multicolumn{1}{c|}{}                        & $0.05$   & $6.8$min      & $42.2$min         & $\star$                          & $1.3\mbox{e-}3$  & $3.4\mbox{e-}3$   & $-$             & $-$                & $-$             \\ \hline
\multicolumn{1}{|c|}{\multirow{2}{*}{$8000$}} & \multicolumn{1}{c|}{\multirow{2}{*}{$804$}} & \multicolumn{1}{c|}{\multirow{2}{*}{$0.04$}} & $0$      & $15.3$min     & $1.4$hr           & $\star$                          & $2\mbox{e-}12$   & $1.1\mbox{e-}5$   & $-$             & $-$                & $-$             \\ \cline{4-12} 
\multicolumn{1}{|c|}{}                        & \multicolumn{1}{c|}{}                       & \multicolumn{1}{c|}{}                        & $0.01$   & $20$min          & $1.4$hr        & $\star$                          & $2.5\mbox{e-}4$  & $4.3\mbox{e-}4$   & $-$             & $-$                & $-$             \\ \hline
\end{tabular}
\end{table*}

To assess the performance of our method, we present simulation results on RGGs over the unit square \cite{BLTYW2006,WZYB2008,CKS2015b,CLS2012}. To generate a RGG, we  uniformly sample $N$ points on the unit square $\left[-0.5,0.5\right]^2$ and fix them to be the sensors. We additionally pick $K$ points at random from the square (distinct from the sensors) and fix them to be the anchors. We assume that the distance between two sensors, or between a sensor and an anchor, is known if it is at most $r$. 

\noindent \textbf{\underline{Experiment 1}:} To understand the importance of rigidity, a network consisting of $500$ sensors and $10$ anchors is considered where $r$ is taken to be $0.17$. We generate several instances of random graphs with the above parameters until we have an instance where the corresponding $\Gamma$ fails to be quasi $3$-connected. We run our algorithm at noise levels $\eta = 0$ and $0.1$ on these instances and record the localization results. We then augment the existing clique system to ensure that $\Gamma$ is quasi $3$-connected. This is done using the heuristics proposed in Section \ref{GP}. We again run our algorithm at noise levels $\eta = 0$ and $0.1$ and record the results. A particular instance is reported in Figure \ref{GammaRigid}. We notice that the proposed algorithm performs poorly if $\Gamma$ is not quasi $3$-connected. In particular, notice that the registration mechanism fails in specific regions of the network. This can be attributed to the ``fold-over'' phenomena associated with patches that are loosely connected to the rest of the patch system \cite{ZLGG2010}. However, when $\Gamma$ is forced to be quasi $3$-connected, we notice that the registration output improves significantly for both the noiseless and noisy cases. In fact, we achieve almost machine-level precision for the noiseless case.

\noindent \textbf{\underline{Experiment 2}:} In Table \ref{SelfCompare}, we report the run-times of the three phases of the algorithm for networks of different sizes (we round $K$ to $10\% $ of $N$ in each case). The experiments were performed using MATLAB $8.2$ on a $4$-core workstation with $3.4$ GHz processor and $32$ GB memory. Notice that the timing of the localization and the partitioning phase increases almost linearly with the number of sensors. Interestingly, the timing does not vary much with the noise level for a fixed $N$. However, the timing of the final registration phase appears to depend heavily on the noise level. An explanation for this is that we use the solution of the spectral relaxation of \eqref{GRET-SDP} as an initialization for the ADMM solver \cite{CKS2015a}. If the spectral relaxation turns out to be a good approximation of the optimal solution, then the ADMM solver converges in few iterations. Else, a large number of ADMM iterations are required. 

\noindent \textbf{\underline{Experiment 3}:} As a final analysis, we study the scaling of ANE with the noise level, when $r =  0.12,0.15, \text{ and } 0.18$. A fixed network consisting of $500$ sensors and $50$ anchors was used. For a particular $\eta$ and $r$, we averaged the ANEs obtained over $10$ realizations of the random graph and the measured distances. The results from a typical experiment are plotted in Figure \ref{Self_perfor}. We notice that the ANE increases almost linearly with $\eta$ when $r =  0.12$ and $0.15$. However, when $r=0.12$, the ANE tends to increase abruptly at large noise levels. The reason for this is that the registration process can fail when $r$ is low and $\eta$ is large (low signal-to-noise ratio scenario). 

\subsection{Comparison}

We now compare the proposed method with the following optimization-based methods: Edge-based Maximum Likelihood (\texttt{E-ML}) relaxation \cite{SL2014}, SNL using SDP (\texttt{SNLSDP}) \cite{BLTYW2006}, Edge-based SDP (\texttt{ESDP}) relaxation \cite{WZYB2008}, SNL using Disk Relaxation (\texttt{SNLDR}) \cite{SXG2015}. \texttt{E-ML} uses distributed optimization to minimize a surrogate of the ML estimator for the distance measurements in SNL. \texttt{SNLDR} uses distributed optimization for a novel convex relaxation of the SNL problem. On the other hand, \texttt{SNLSDP} is a centralized algorithm which is based on an SDP-based relaxation of the SNL problem. \texttt{ESDP} is a further relaxation of \texttt{SNLSDP} that can handle large networks.

\noindent \textbf{\underline{Experiment 4}:} The proposed method is compared with \texttt{E-ML}, \texttt{SNLSDP}, \texttt{ESDP} and \texttt{SNLDR} on random geometric graphs. The results are reported in Table \ref{Compare}. For a fair comparison with \texttt{SNLDR}, we additionally placed an anchor at each of the four corners of the unit square to ensure that the sensors are in the convex hull of the anchors \cite{SXG2015}. The localization obtained using our method is comparable with that obtained from \texttt{SNLSDP} for small networks ($N \leq 500$). The performance of \texttt{SNLSDP} starts degrading when $N > 500$, and it cannot handle large networks ($N>1000$). On the other hand, the proposed method is able to maintain its performance across networks of all sizes. Notice that, though \texttt{ESDP} can scale up to networks of size $8000$, its performance  falls off abruptly when $N>4000$.  The proposed method is generally faster than \texttt{ESDP} and \texttt{SNLSDP}.

\begin{table*}[!htbp]
\centering
\caption{Visual comparison of  the proposed algorithm with \texttt{SNLSDP} \cite{BLTYW2006}, \texttt{ESDP} \cite{WZYB2008} and \texttt{SNLDR} \cite{SXG2015} for a random geometric graph on the unit square consisting of $100$ sensors and $14$ anchors. The radio-range used is $r=0.4$. See Figure \ref{GammaRigid} for a description of the symbols \textcolor{green}{$\circ$}, \textcolor{red}{$\star$}, and \textcolor{blue}{$\Diamond$}. The (ANE, Run-time) are reported in the caption.}
\label{my-labe1}
\begin{tabular}{|c|m{3.8cm}|m{3.8cm}|m{3.8cm}|m{3.8cm}|}
\cline{1-5}
$\eta$        & \hspace{1.4cm} Proposed & \hspace{1.4cm} \texttt{SNLSDP} \cite{BLTYW2006} & \hspace{1.4cm} \texttt{ESDP} \cite{WZYB2008} & \hspace{1.4cm} \texttt{SNLDR} \cite{SXG2015} \\ \hline
 $0$  &    \raisebox{-\totalheight}{\includegraphics[width= \linewidth]{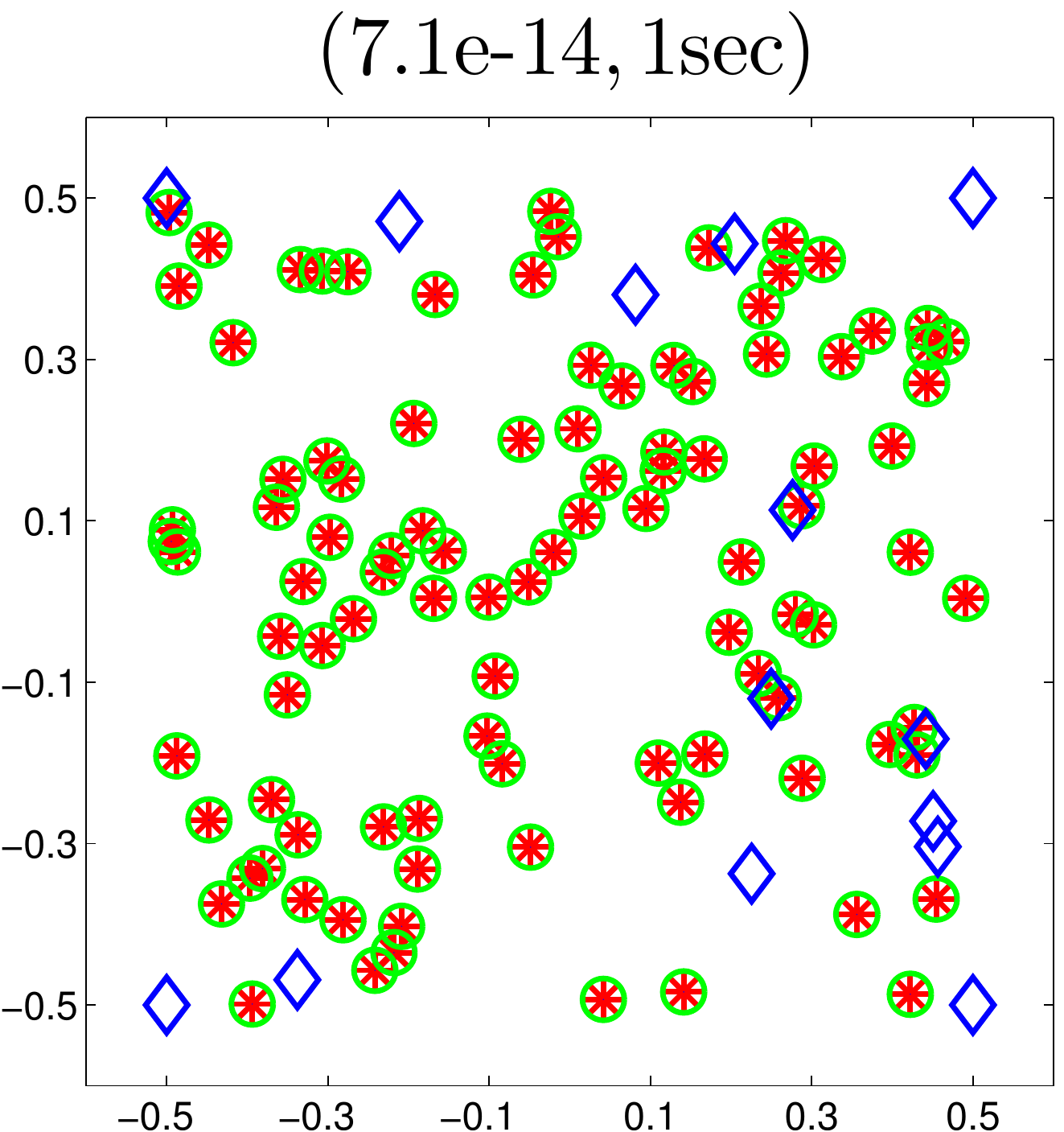}} \vspace{0.5mm}&     \raisebox{-\totalheight}{\includegraphics[width= \linewidth]{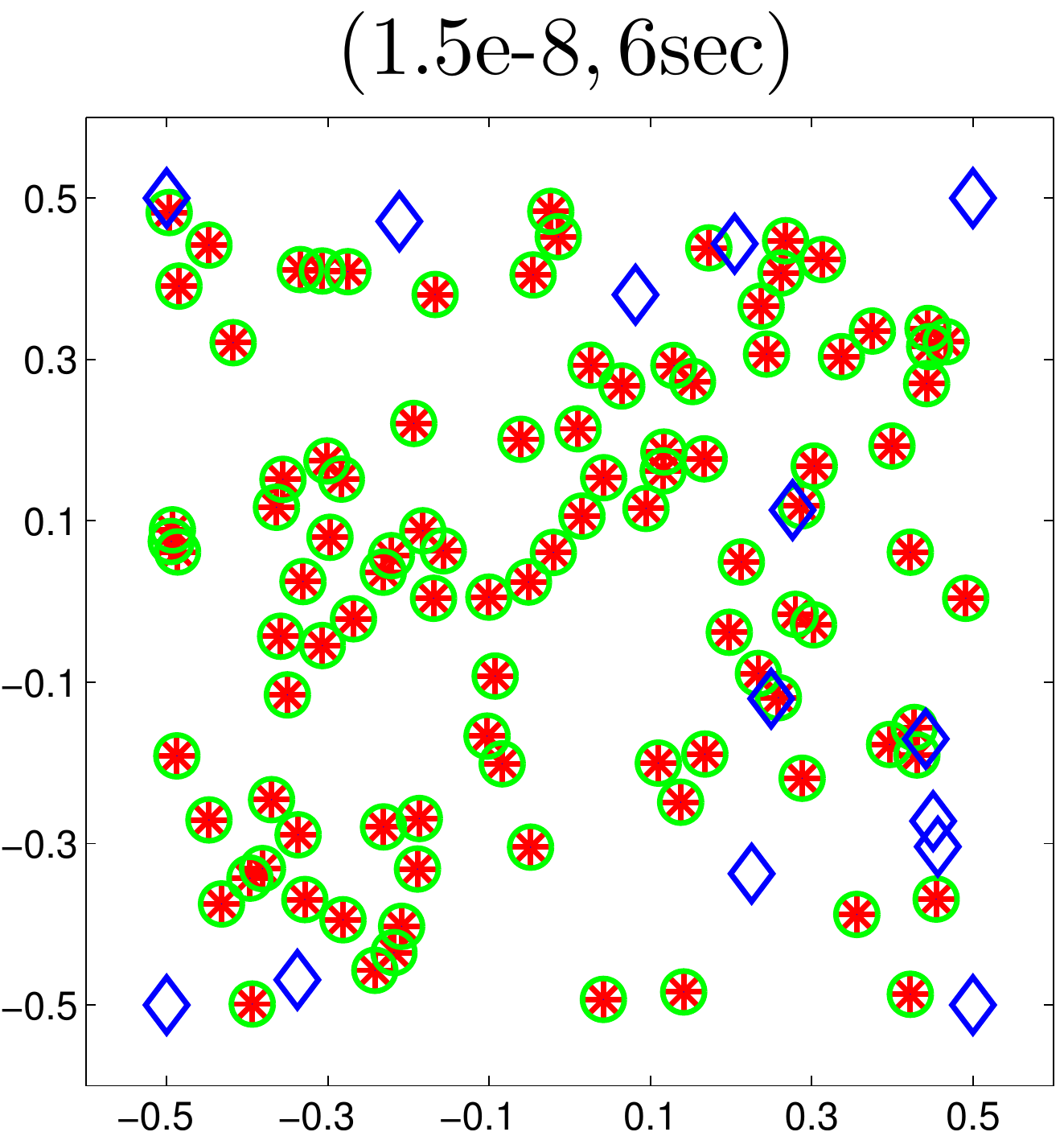}} \vspace{0.5mm} &                                  \raisebox{-\totalheight}{\includegraphics[width= \linewidth]{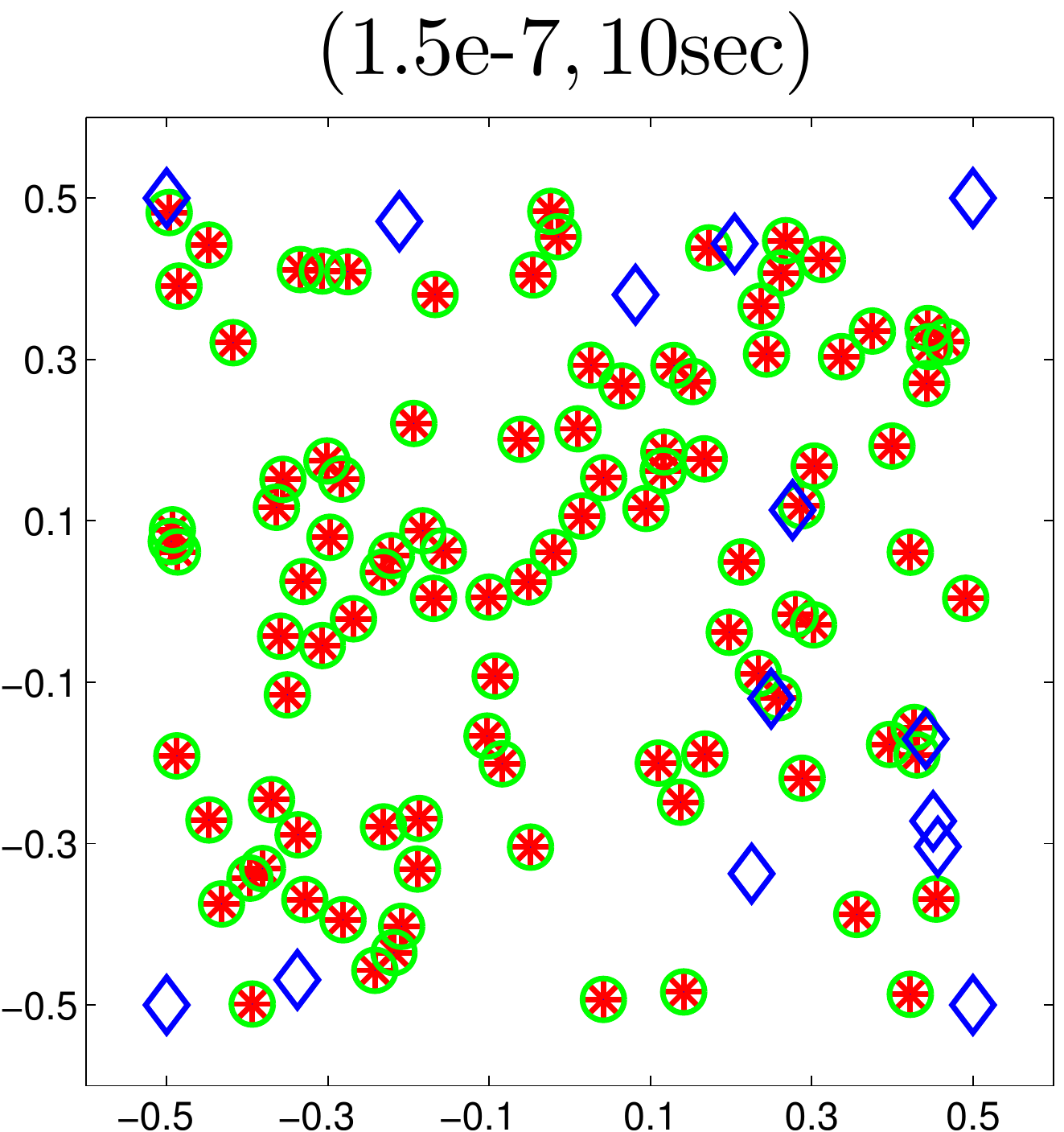}}  \vspace{0.5mm} & \raisebox{-\totalheight}{\includegraphics[width= \linewidth]{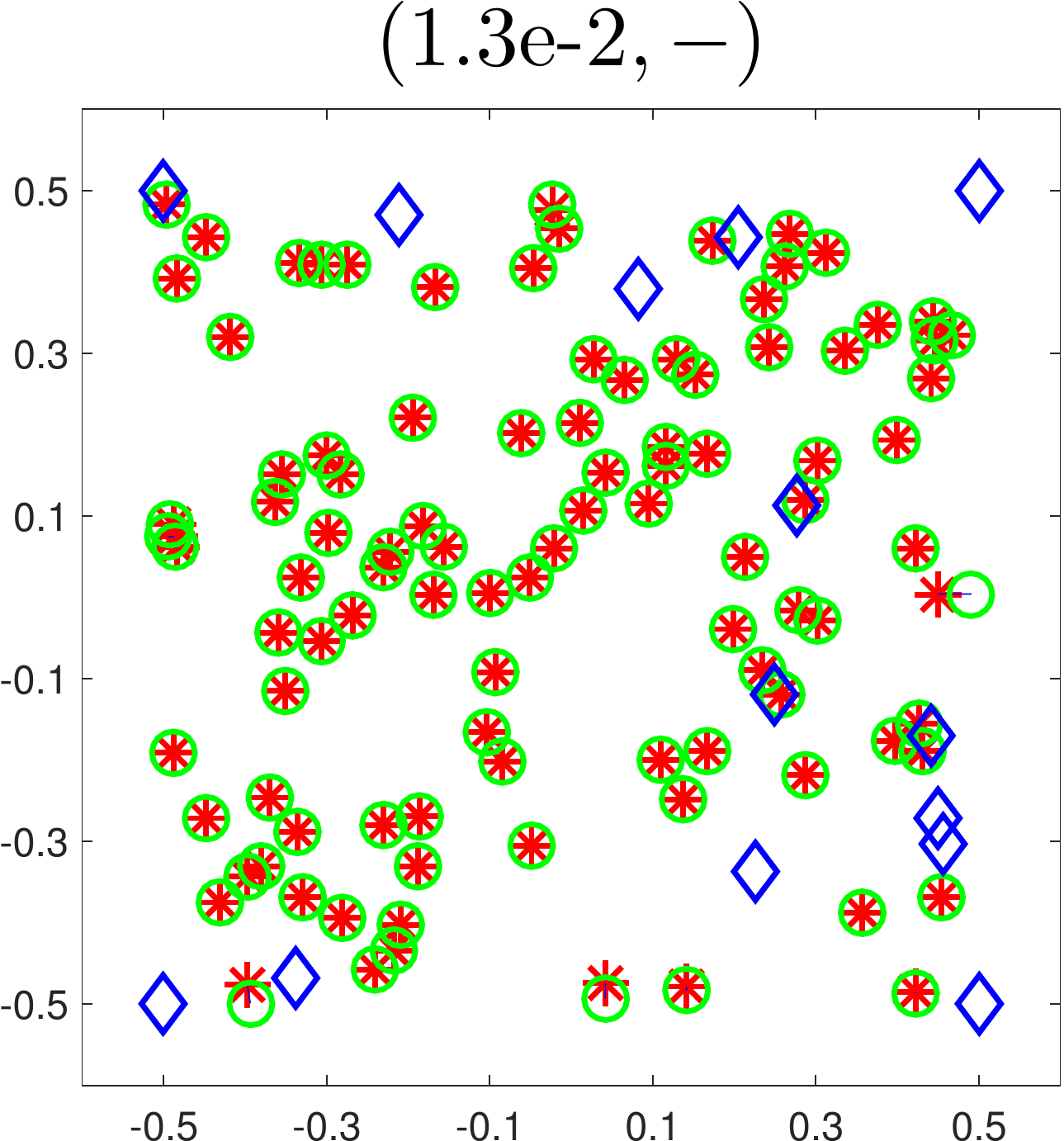}} \vspace{0.5mm} \\ \cline{1-5} 
 $0.1$  &    \raisebox{-\totalheight}{\includegraphics[width= \linewidth]{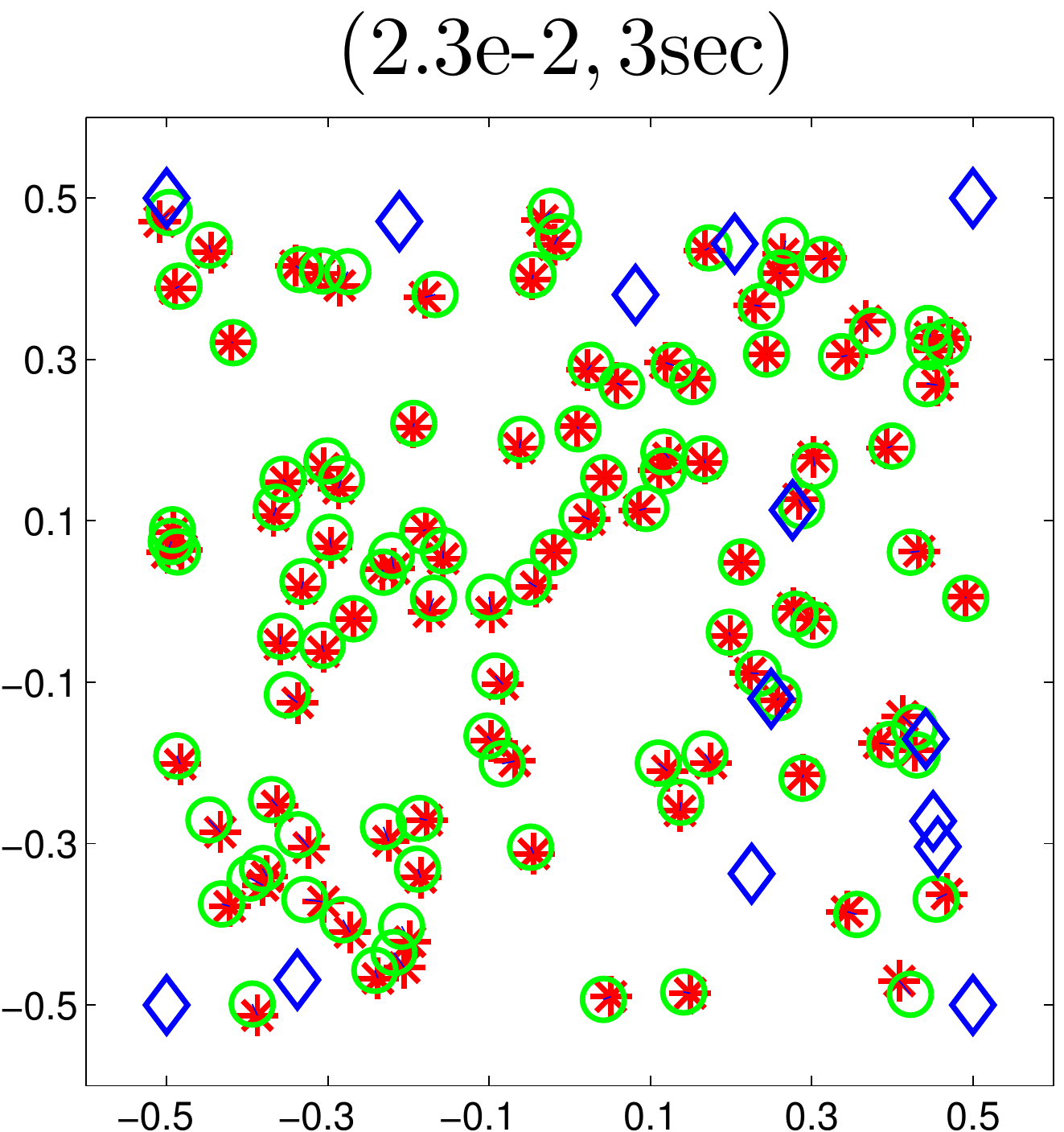}} \vspace{0.5mm} &     \raisebox{-\totalheight}{\includegraphics[width= \linewidth]{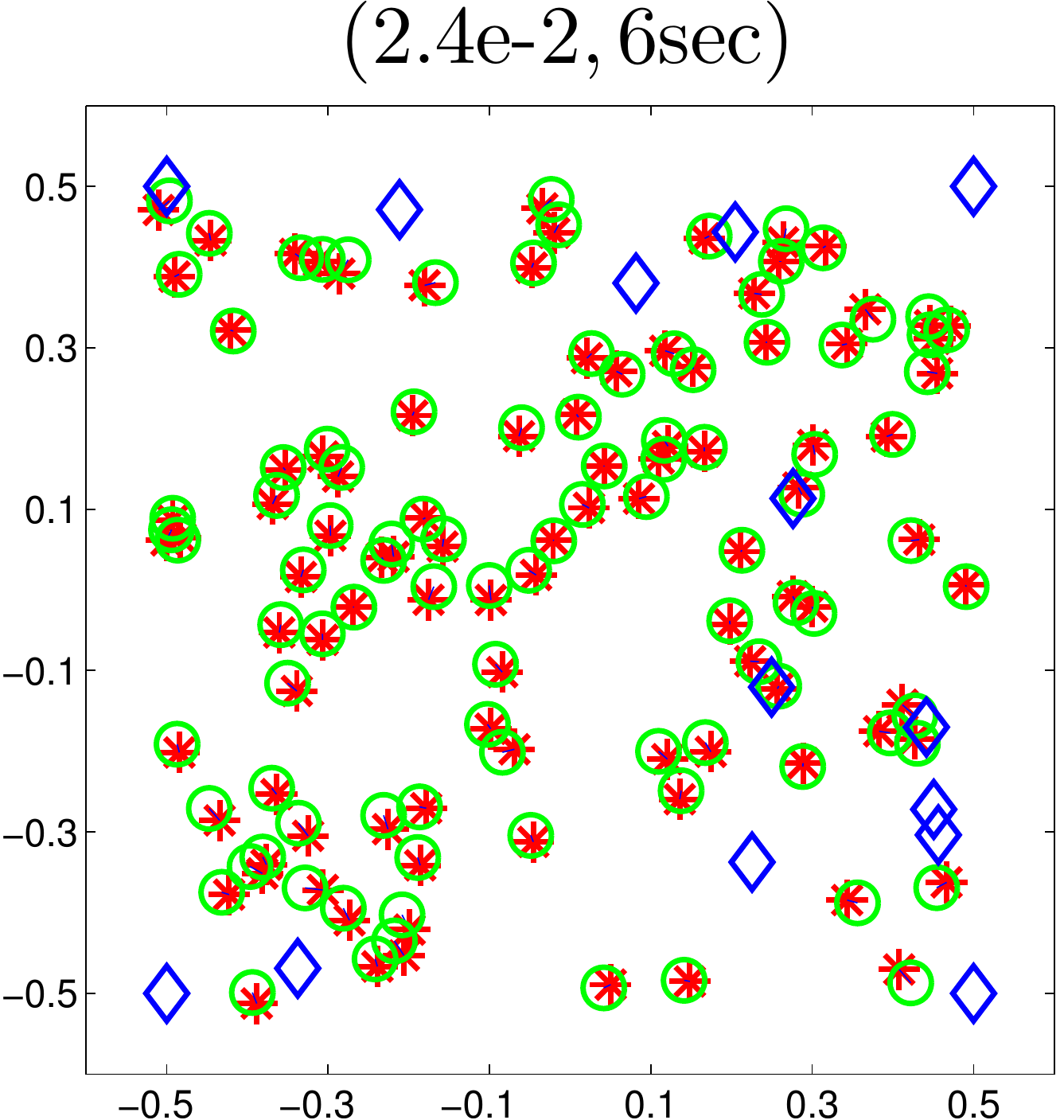}} \vspace{0.5mm}    &                                  \raisebox{-\totalheight}{\includegraphics[width= \linewidth]{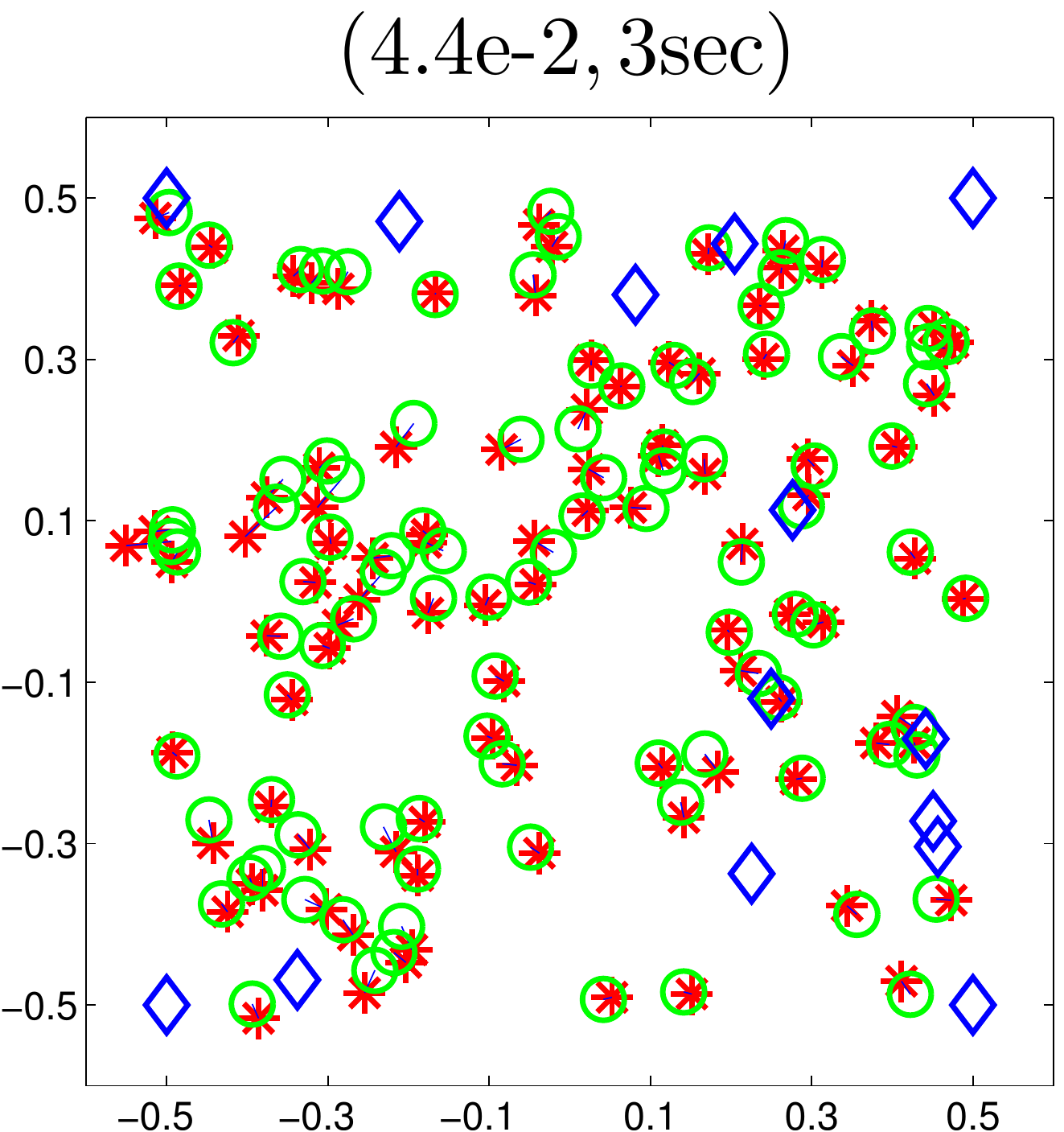}} \vspace{0.5mm} &        \raisebox{-\totalheight}{\includegraphics[width= \linewidth]{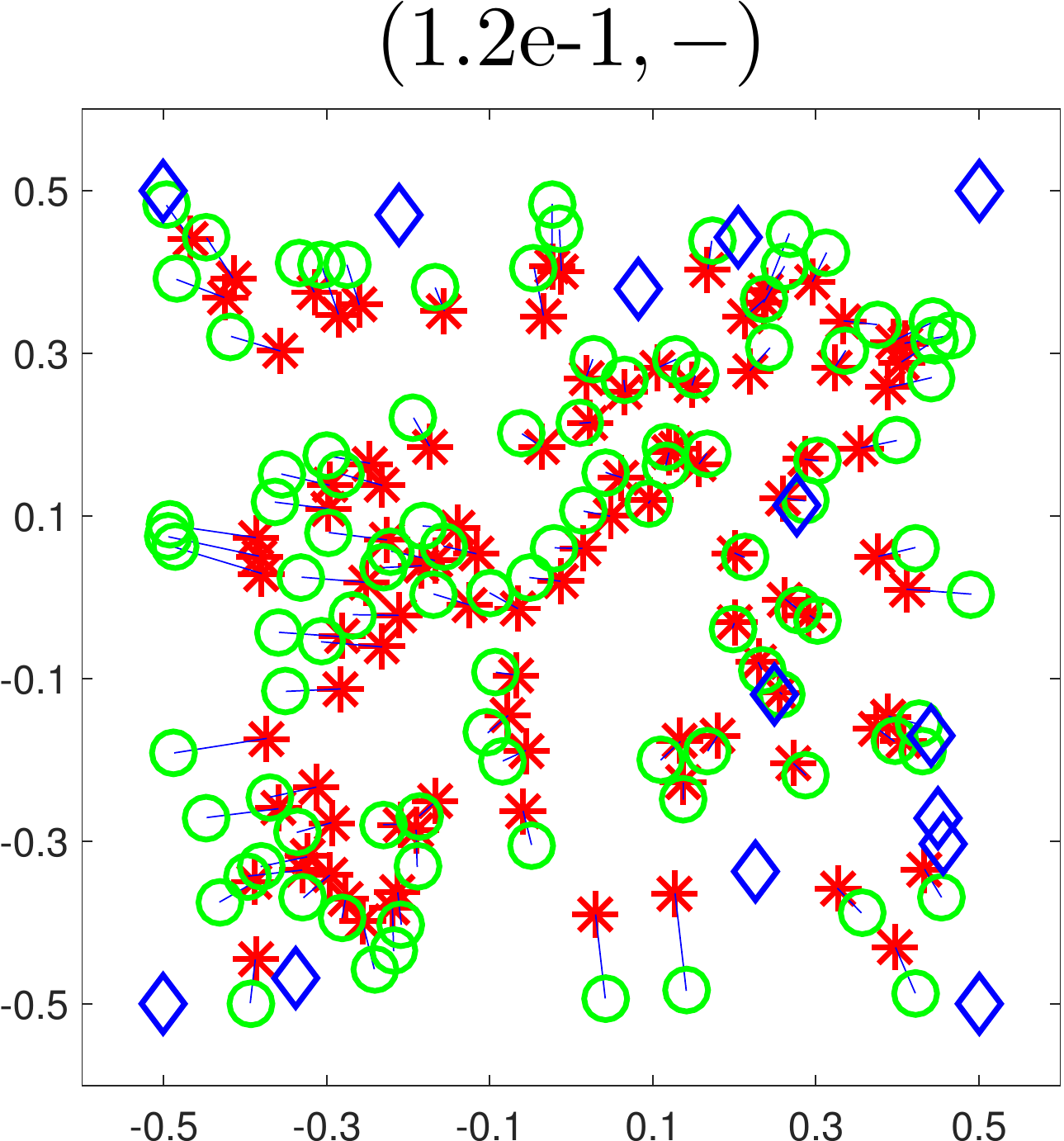}} \vspace{0.5mm} \\ \hline
\end{tabular}
\end{table*}

\begin{table*}[!htbp]
\centering
\caption{Comparison of localization results for the PACM logo \cite{CLS2012}. The parameters are $N=382,K=43$, and $r=1.9$. Blue circles (\textcolor{blue}{$\circ$}) denote original (column 1) and reconstructed (columns 2-5) sensor locations, and red diamonds (\textcolor{red}{$\Diamond$}) denote anchor locations.}
\label{my-labe2}
\begin{tabular}{|c|m{3cm}|m{3cm}|m{3cm}|m{3cm}|m{3cm}|}
\cline{1-6}
$\eta$     & \hspace{0.8cm} Original   & \hspace{0.8cm} Proposed  & \hspace{0.8cm} \texttt{SNLSDP} \cite{BLTYW2006} & \hspace{0.8cm} \texttt{ESDP} \cite{WZYB2008} & \hspace{0.8cm} \texttt{SNLDR} \cite{SXG2015} \\ \hline
 $0$    & \raisebox{-\totalheight}{\includegraphics[width= \linewidth]{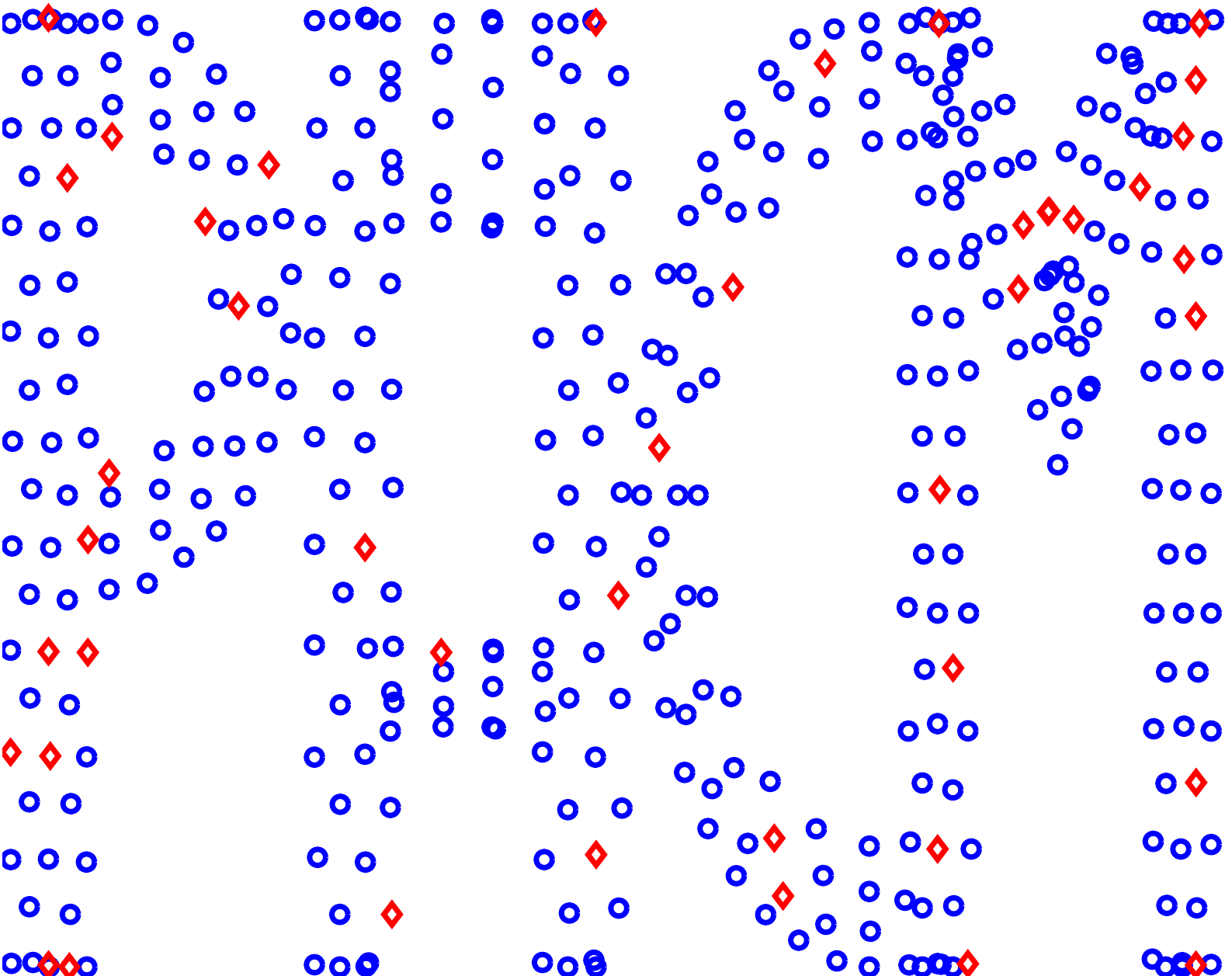}} \vspace{2mm}    &   \raisebox{-\totalheight}{\includegraphics[width= \linewidth]{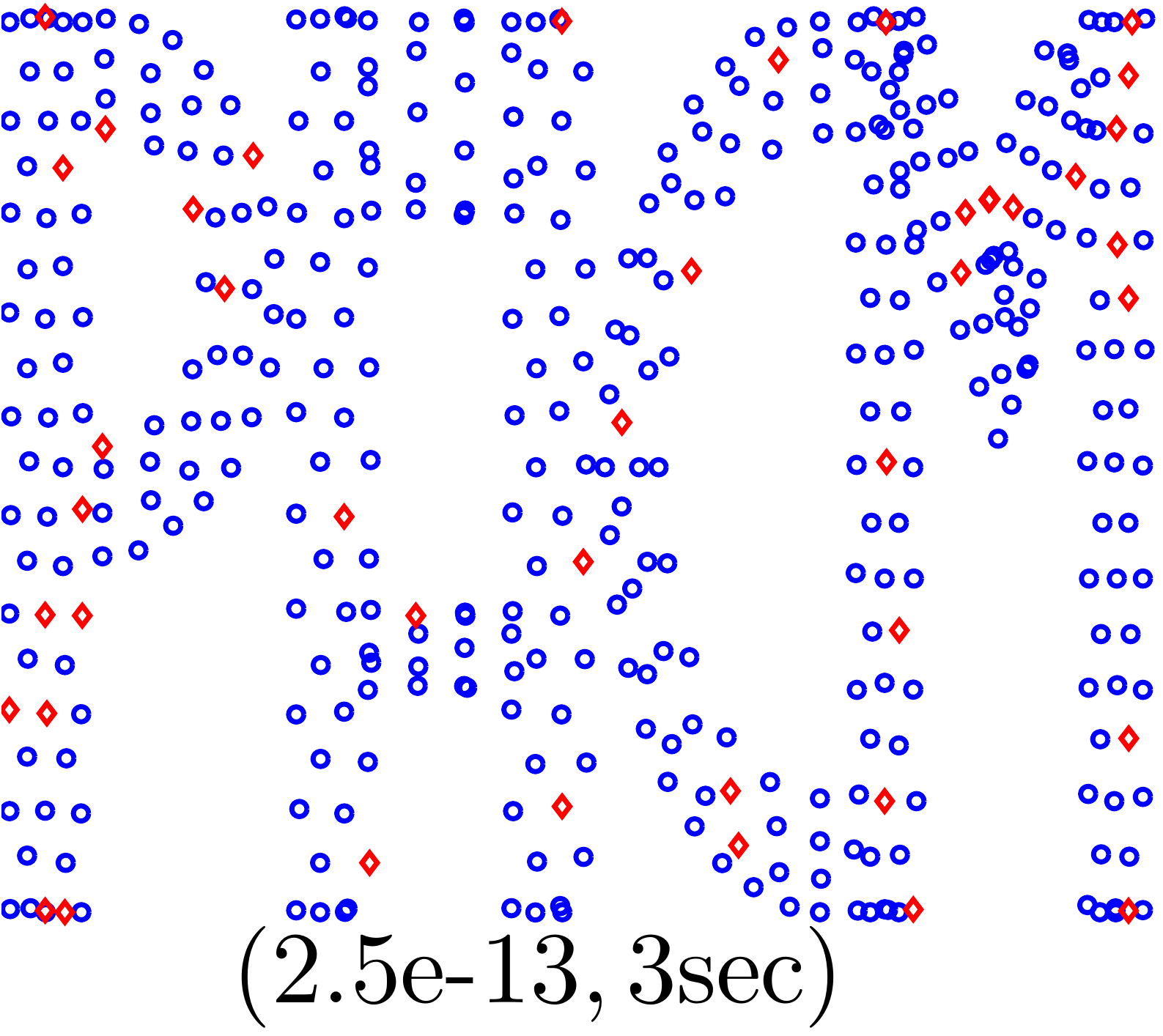}}              &                                     \raisebox{-\totalheight}{\includegraphics[width= \linewidth]{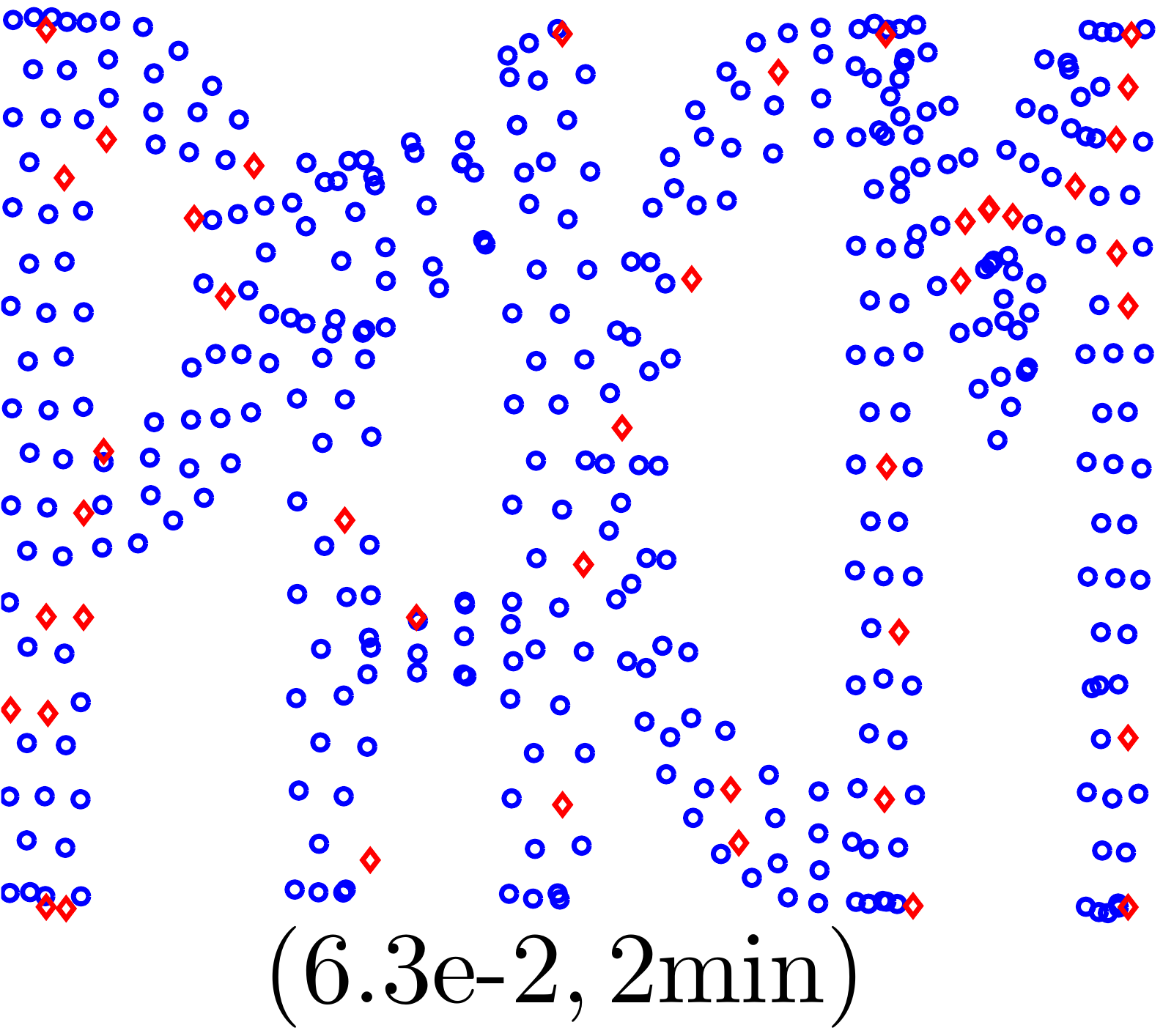}}   &    \raisebox{-\totalheight}{\includegraphics[width= \linewidth]{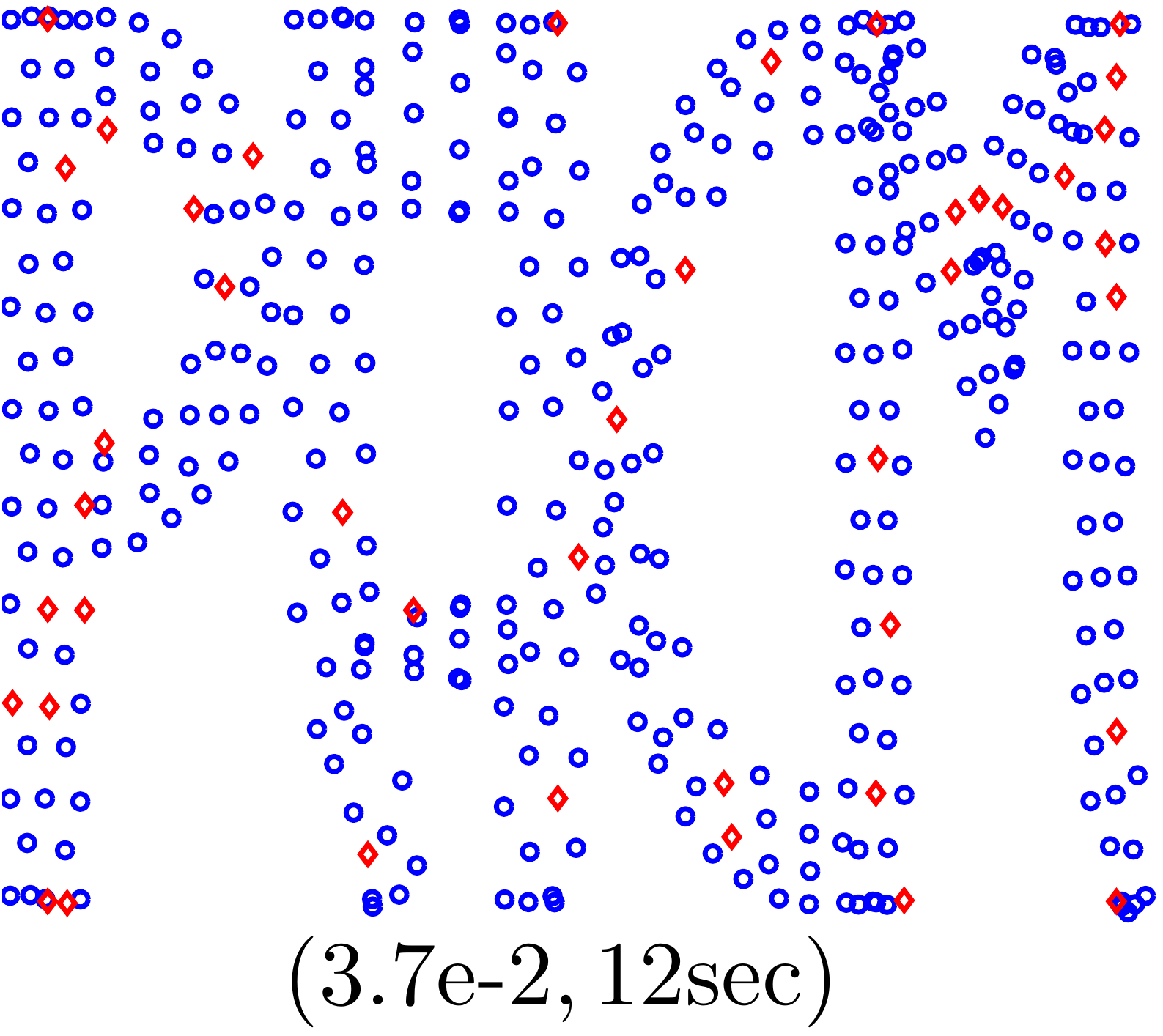}}                                &          \raisebox{-\totalheight}{\includegraphics[width= \linewidth]{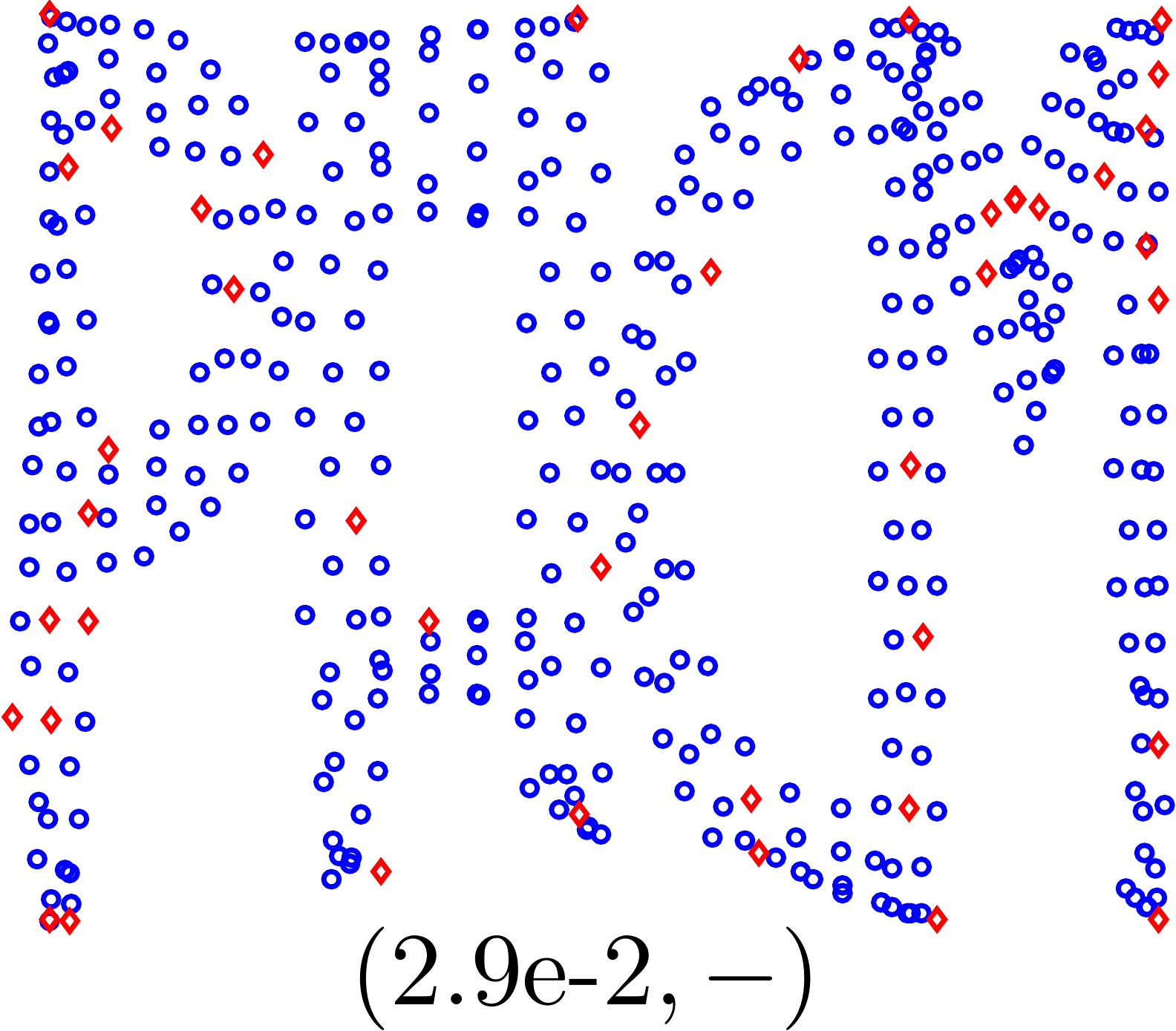}}                         \\ \cline{1-6} 
 $0.5$ &    \raisebox{-\totalheight}{\includegraphics[width= \linewidth]{figures/PACM_Orig-eps-converted-to.pdf}} \vspace{2mm}    &   \raisebox{-\totalheight}{\includegraphics[width= \linewidth]{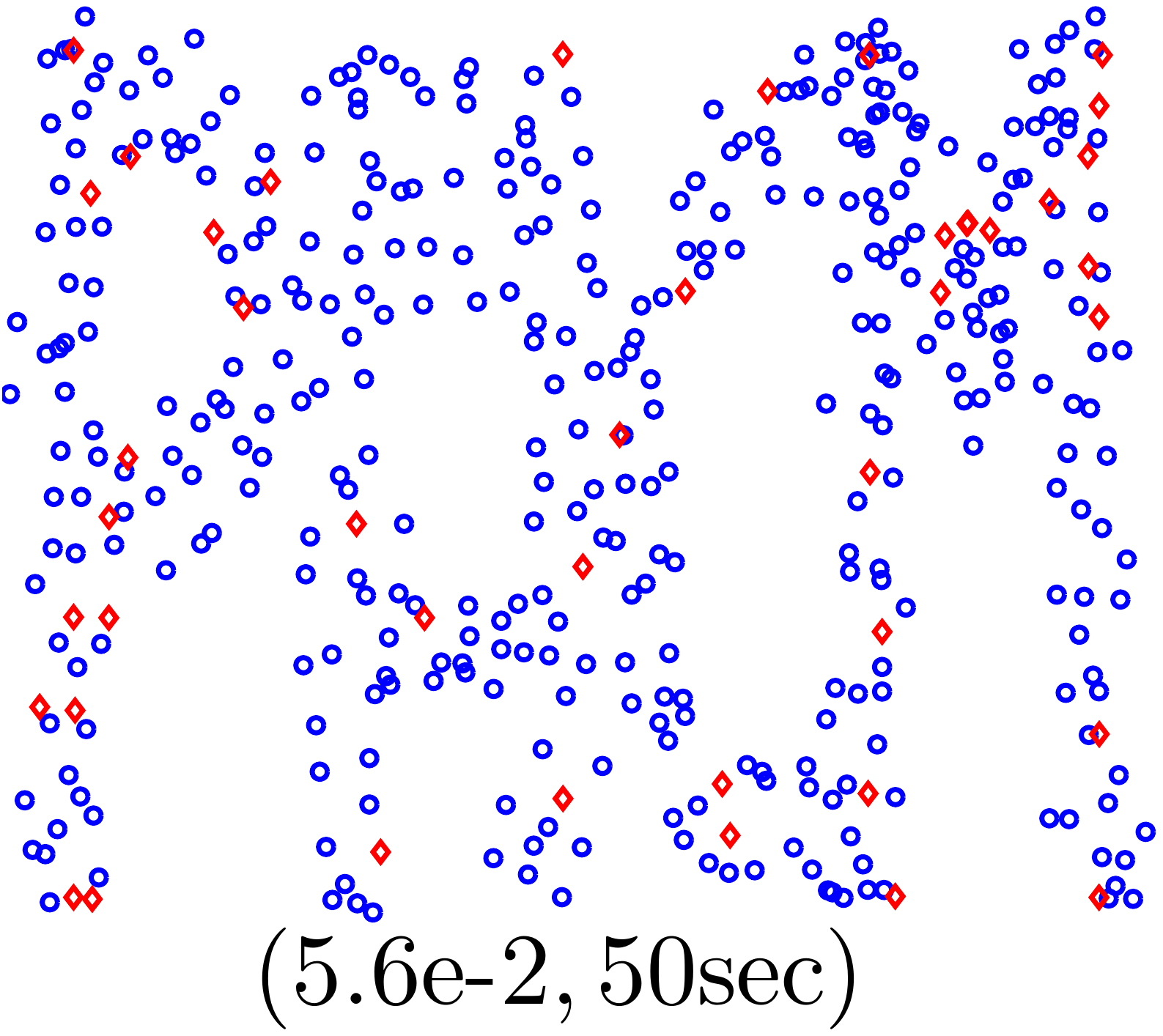}}  &                                     \raisebox{-\totalheight}{\includegraphics[width= \linewidth]{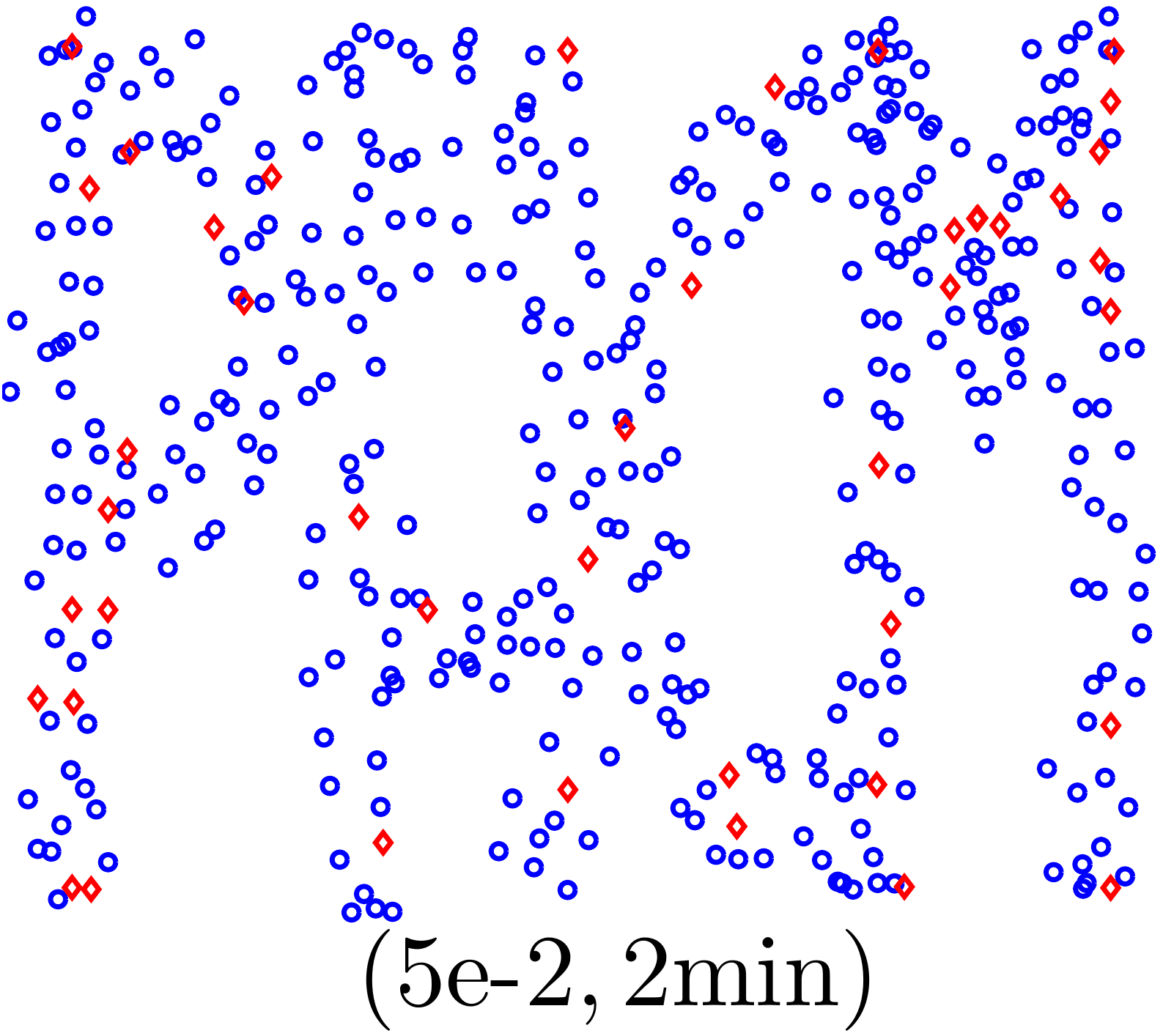}} &    \raisebox{-\totalheight}{\includegraphics[width= \linewidth]{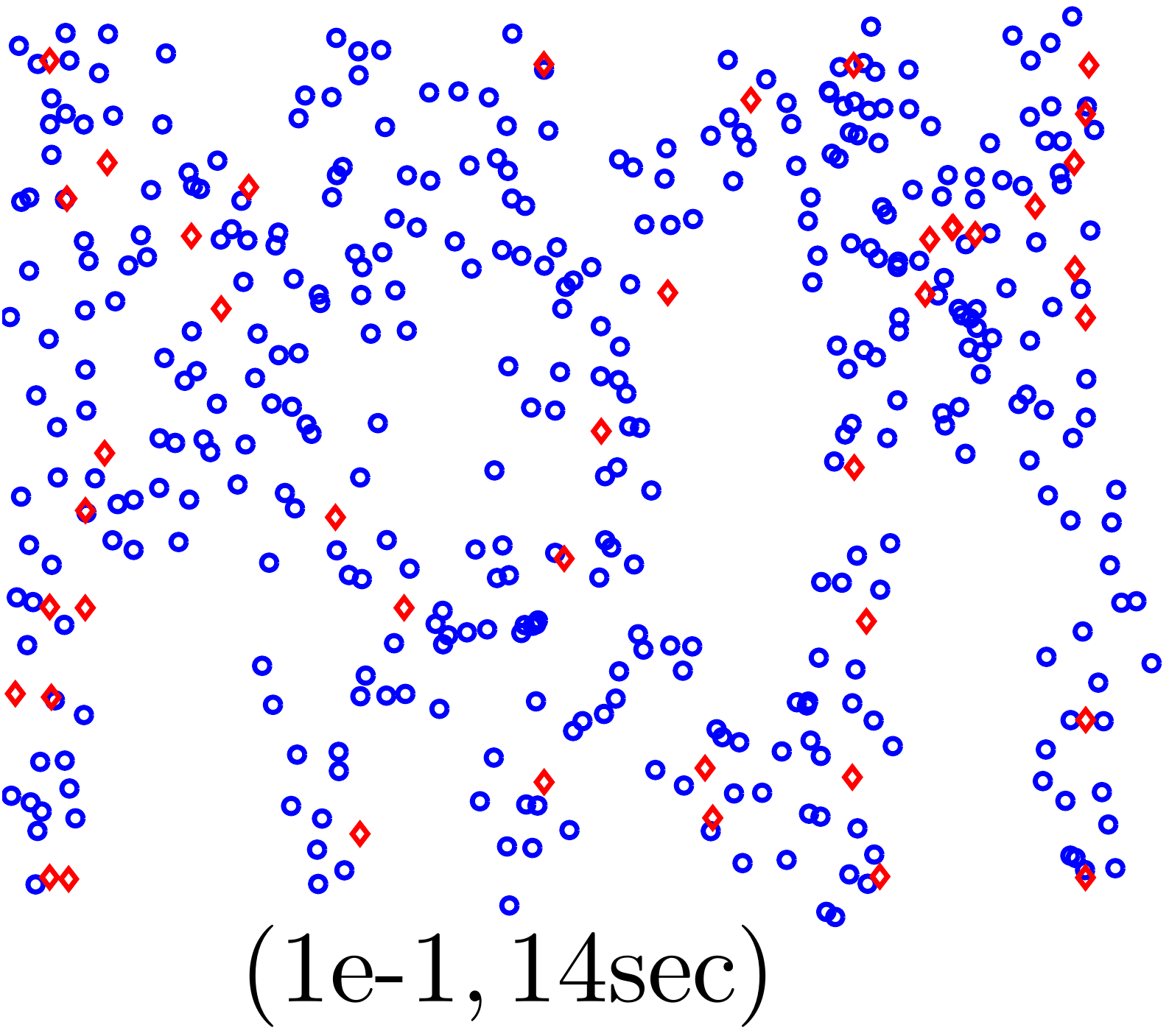}}  &          \raisebox{-\totalheight}{\includegraphics[width= \linewidth]{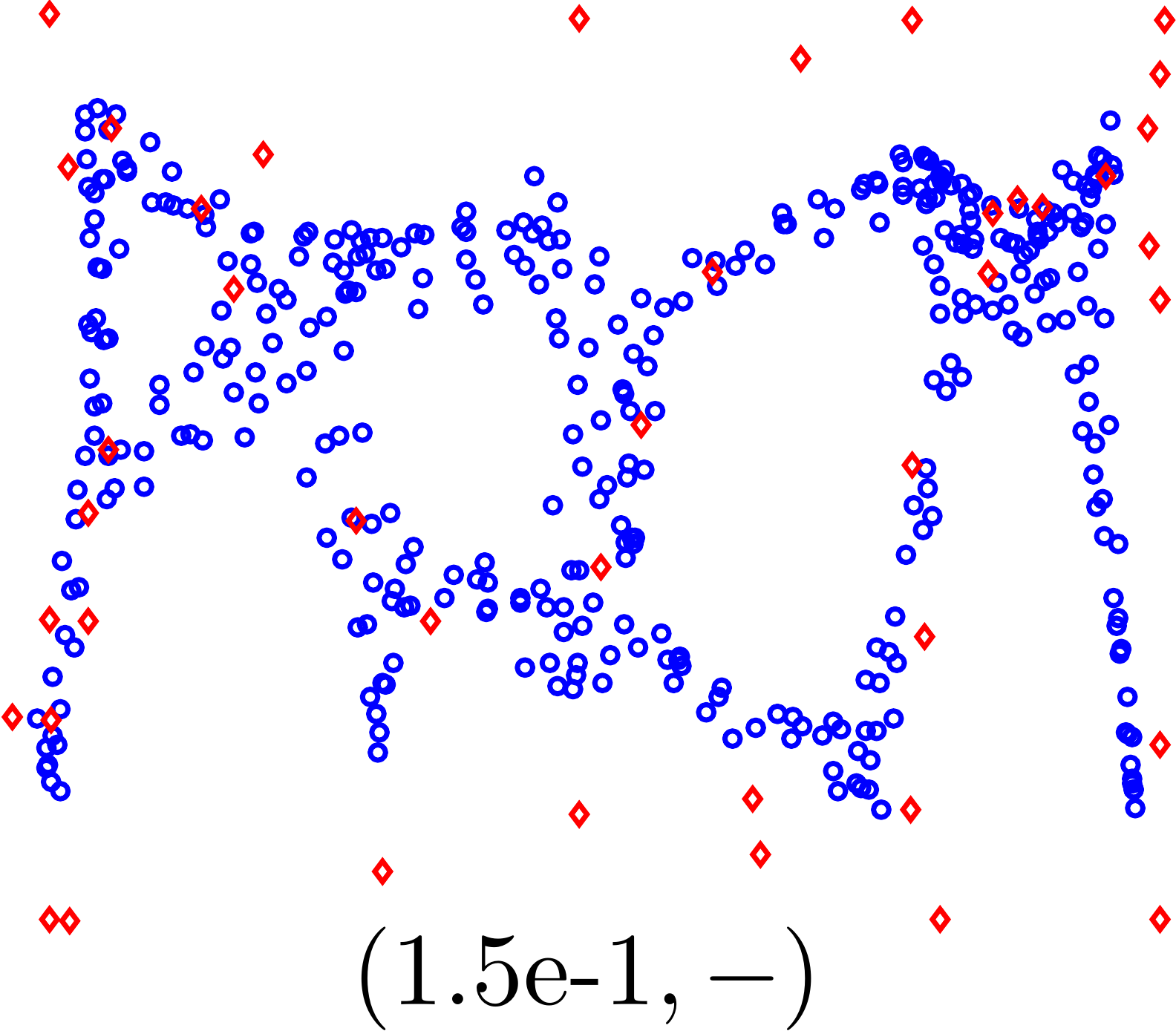}}  \\ \hline
\end{tabular}
\end{table*}

\noindent \textbf{\underline{Experiment 5}:} We provide some visual comparison in Figures \ref{my-labe1} and \ref{my-labe2} for RGGs and the PACM logo \cite{CLS2012}. The latter consists of $425$ points sampled from the logo. We randomly set $43$ points as anchors. Notice that the reconstruction from the proposed method is visibly superior to the competing methods in either case, which is also reflected by the ANE. The accuracy is competitive with \texttt{SNLSDP}, but consistently better than the other methods. In particular, notice the poor localizations obtained using \texttt{SNLDR} when $\eta=0.5$.

Additional comparisons with \cite{SXG2015}, \cite{BLTYW2006},  \cite{WZYB2008},  and \cite{ADPV2012} are provided in the supplementary material. The MATLAB code of our algorithm is publicly available \cite{SJC2017}.

\section{Conclusion}
\label{conclusion}

We demonstrated that by transforming the localization problem into a registration problem, one can achieve scalability without compromising the localization accuracy. In particular, the convex relaxation of the registration problem appears to be better behaved in terms of scalability and approximation quality compared to the convex relaxations of the localization problem. For example, the proposed algorithm can localize a  network of $8000$ nodes in $15$ minutes with almost machine-precision accuracy of  $1\mbox{e-}12$. In contrast, the convex relaxation in \cite{BLTYW2006} cannot be scaled beyond $1000$ nodes. An exception in this regard is \texttt{ESDP}, which can be scaled to networks with thousands of nodes. However, its localization accuracy starts falling off with the increase in network size. A key contribution of the paper is that we formulated and analysed the rigidity problem associated with multi-patch registration. An open question that emerged from this analysis is whether quasi-connectivity is sufficient for the patch configuration to be rigid. Another relevant question that remains unaddressed is the impact of rigidity on the performance of the registration algorithm, both in terms of tightness and stability. These  will be investigated in future work.

\section{Supplementary}

\subsection{Proof of Theorem II.8}
\label{ProofTheorem}
In this section, we prove Theorem II.8. First, we recall a basic assumption that was made in this regard.
\begin{assumption}[Non-degeneracy]
\label{assumption}
There are at least $d+1$ non-degenerate points in each patch.
\end{assumption}
We now restate Theorem II.8.
\begin{theorem}[Necessary condition]
\label{THEOREM}
Under Assumption \ref{assumption}, if a configuration is rigid in $\mathbb{R}^d$, then its correspondence graph must be quasi $(d\!+\!1)$-connected. 
\end{theorem}
To prove Theorem \ref{THEOREM}, we will need the following proposition:
\begin{proposition}[]
\label{Menger}
The following are equivalent.\\
(a) The correspondence graph $\Gamma$ is quasi $k$-connected.\\
(b) $\cE(\Gamma)$ can be divided into two disjoint subsets $E_1$ and $E_2$ such that the edges from $E_1$ and that from $E_2$ are\\
\indent (i)  incident on at least $k$ common vertices from $\cV_1(\Gamma)$, and\\
\indent (ii) not incident on any common vertex from $\cV_2(\Gamma)$.
\end{proposition}
For completeness, we recall Problems II.1 and II.2 from the main manuscript. 
\begin{problem}[Registration]
\label{prob1}
Find $\x_1, \dots , \x_N$ and rigid transforms $\cQ_1, \dots , \cQ_M$ such that, for $1 \leq i \leq M$,
\begin{equation}
\label{prob6}
\x_k = \cQ_i (\x_{k,i}) \quad \text{and}  \quad \bar{\a}_l = \cQ_i (\bar{\a}_l),
\end{equation}
where $k \in \cC_i \setminus \cA$ and $l \in \cC_i \cap \cA$.
\end{problem}
\begin{problem}[Uniqueness]
\label{prob2} 
Determine whether Problem \ref{prob1} have a unique solution up to a rigid transform. That is, if $\x_1, \dots , \x_N$  is a solution of Problem \ref{prob1}, then is it necessary that for some rigid transform $\cR$,
\begin{equation*}
\label{des_reg_out1}
\x_k = \cR ( \bar{\x}_k) \quad \text{ and } \quad \bar{\a}_l = \cR (\bar{\a}_l),
\end{equation*}
where $k \in \cS$ and $l \in \cA$?
\end{problem}
Moreover, we assume that the points in each patch have been derived from the respective points in $\cX_s \cup \cX_a$ via a rigid transform. 
Let
\begin{equation}
 \label{LocGlob1}
\bar{\x}_k =\cR_i(\x_{k,i}) = \bO_i \x_{k,i} + \t_i  \qquad (k \in \cC_i \backslash \cA),
\end{equation}
and
\begin{equation}
 \label{LocGlob2}
\bar{\a}_l =\cR_i(\bar{\a}_l) \qquad (l \in \cC_i \cap \cA),
\end{equation}
We consider a different registration problem where the patch coordinates are replaced by the original coordinates.
\begin{problem}[Registration]
\label{prob4} Find $\x_1, \dots , \x_N$ and rigid transforms $\cT_1, \dots , \cT_M$ such that, for $1 \leq i \leq M$,
\begin{equation}
\label{prob7}
	\x_k = \cT_i (\bar{\x}_{k}) \quad \text{and} \quad \bar{\a}_l = \cT_i (\bar{\a}_l),
\end{equation}
where $k \in \cC_i \setminus \cA$ and $l \in \cC_i \cap \cA$.
\end{problem}
A trivial solution is $\x_k = \bar{\x}_{k}$ and $\cT_i= (\bI_d, \mathbf{0})$. As with Problem \ref{prob2}, we can ask whether this is the only solution. It turns out that the  questions are related.
\begin{proposition}[Equivalence]
\label{prob5}
Problem \ref{prob1} has an unique solution if and only if Problem \ref{prob4} has an unique solution.
\end{proposition}
\begin{proof}
Combining \eqref{LocGlob1}, \eqref{LocGlob2} and \eqref{prob7}, we can write
\begin{equation}
\label{prob8}
\x_k = (\cT_i \circ \cR_i) (\x_{k,i})  \quad \text{and} \quad \bar{\a}_l =  (\cT_i \circ \cR_i) ( \bar{\a}_l ).
\end{equation}
Comparing \eqref{prob8} with \eqref{prob6}, we have $\cQ_i=\cT_i \circ \cR_i$. It follows that the $\cT_i$'s are unique if and only if the $\cQ_i$'s are  unique. Moreover, the uniqueness of the $\x_k$'s follows from the uniqueness of the transforms and relations \eqref{prob6} and \eqref{prob7}.
\end{proof}
We also make the observation concerning Problem \ref{prob4} that $\cT_1, \dots , \cT_M$ satisfying \eqref{prob7} are unique, i.e., $\cT_i=\cR(\bI_d, \mathbf{0})$ for some rigid transform $\cR$, if and only if the corresponding $\x_1,\ldots,\x_N$ are related to $\bar{\x}_1,\ldots,\bar{\x}_N$ via a rigid transform. If $\cT_i=\cR(\bI_d, \mathbf{0})$, then it follows from \eqref{prob7} that $\x_k =\cR (\bar{\x}_k)$. Conversely, if  $\x_k =\cR (\bar{\x}_k)$ for some rigid transform $\cR$, then the corresponding $\cT_i$ should necessarily be of the form $\cT_i=\cR(\bI_d, \mathbf{0})$. Indeed, if some $\cT_i \neq \cR(\bI_d, \mathbf{0})$, then we can construct a solution that is not related to $\bar{\x}_1,\ldots,\bar{\x}_N$ via a rigid transform, and this would lead to a contradiction.  

To complete the proof of Theorem \ref{THEOREM}, it remains to show that, if the solution of  Problem \ref{prob4} is unique, then $\Gamma$ must be quasi $(d+1)$-connected. We will prove this by contradiction. As a first step, we note that  $\Gamma$ is at least quasi-$1$ connected.
\begin{proposition}
\label{conn}
If Problem \ref{prob4} has a unique solution, then $\Gamma$ must be quasi-$k$ connected for some $k \geq 1$.
\end{proposition}
\begin{proof}
Indeed, suppose that there exist non-empty subsets $S$ and $T$ of $\cV_2(\Gamma)$ such that there is no path between any $i \in S$ and $j \in T$. Define
\begin{equation*}
A = \bigcup_{\alpha \in S} \cC_{\alpha} \quad \text{ and } \quad B= \bigcup_{\beta \in T} \cC_{\beta}.
\end{equation*}
Clearly, $A \cap B$ must be empty. Else, we can find a path between some $i \in S$ and $j \in T$, which would violate our assumption. However, on setting $\cT_i=(\bI_d,\mathbf{0})$ for $i \in S$, and $\cT_j=(-\bI_d, \mathbf{0})$ for $j \in T$, we obtain a solution to Problem \ref{prob4} which is different from the trivial solution. Hence, our assumption about the existence of $S$ and $T$ must be wrong.
\end{proof}
In fact, we can make the stronger claim that $\Gamma$ is quasi $k$-connected, where $k \geq d+1$. To establish the claim, we show that the rigidity assumption is violated if $k \leq d$. 

First, we introduce few notations about paths. Suppose that there are one or more paths between two vertices of $\cV_2(\Gamma)$. We denote the $j$-th vertex on the $i$-th path using  $\sigma^j_i$. In particular, $\sigma_i^1$ and $\sigma_i^{p_i}$ are the initial and final vertices, where $p_i$ is the number of vertices on the path. Since $\Gamma$ is bipartite, $p_i$ must be odd, and 
\[ \sigma_i^j \in \begin{cases} 
      \cV_1 \left(\Gamma\right) & \text{for} \ j = 2,4,\dots,p_i-1, \\
      \cV_2 \left(\Gamma\right) & \text{for} \ j = 1,3,\dots,p_i. \\
   \end{cases}
\]
For $1 \leq j \leq (p_i-1)/2$, consider the vertices  
\begin{equation*}
\sigma_i^{2j},\  \sigma_i^{2j-1}, \text{ and } \sigma_i^{2j+1}.
\end{equation*}
The first vertex represents a node, while the latter two represent patches. Moreover, the node belongs to both the patches. Therefore, for $1 \leq i \leq k$ and $1 \leq j \leq (p_i-1)/2$,
\begin{equation}
\label{Theo_prof_1}
\bO_{\sigma_i^{2j-1}}\bar{\x}_{\sigma_i^{2j}}+\t_{\sigma_i^{2j-1}} = \bO_{\sigma_i^{2j+1}}\bar{\x}_{\sigma_i^{2j}}+\t_{\sigma_i^{2j+1}}.
\end{equation}
To arrive at a contradiction, we show that if $k \leq d$, then there exists at least some $\cT_i=(\bO_i,\t_i),1\leq i \leq M,$ different\footnote{Without loss of generality, we omit the global rigid transform $\cR$.} from $(\bI_d,\bm{0})$ for which the system of equations in \eqref{prob7} hold. To do so, we divide $\cV_2(\Gamma)$ into two disjoint sets. Note that, from Proposition \ref{Menger}, we can identify disjoint subsets $E_1,E_2 \subset \cE(\Gamma)$ such that the edges from $E_1$ and that from $E_2$ are not incident on any common vertex of $\cV_2(\Gamma)$. In particular, define $S \subset \cV_2(\Gamma)$ to  be the vertices on which  the edges of $E_1$ are incident. Similarly, let $T \subset \cV_2(\Gamma)$ be the vertices on which  the edges of $E_2$ are incident. Then, $S$ and $T$ are non-empty and disjoint. Without loss of generality, we assume that the vertex corresponding to the anchor patch belongs to $S$. Since $\Gamma$ is quasi $k$-connected, we can find a distinct vertex $t \in T$ which is connected with the anchor patch vertex by paths $\sigma_1,\ldots,\sigma_k$ that are $\cV_1(\Gamma)$-disjoint. 

Note that, since the anchor patch is fixed, $\bO_{M+1}=\bI_d$ and $\t_{M+1}=\bm{0}$. Therefore, we set
\[ \cT_i= (\bO_i,\t_i) = \begin{cases} 
   (\bO,\t) , & \text{ if } \ i \in T, \\
   (\bI_d,\mathbf{0}) & \text{ if } \ i \in S,
	\end{cases}
\]
and show that if $k \leq d$, then we can find $(\bO,\t) \neq (\bI_d,\bm{0})$ such that \eqref{prob7} holds. Note that Proposition \ref{Menger} also tells us that the edges from $E_1$ and that from $E_2$ are incident on exactly $k$ common vertices from $\cV_1(\Gamma)$; we denoted these vertices using $\Omega$. It is also be reasoned that each path contains exactly one vertex from $\Omega$. Assume that $\sigma_i^{2q_i} \in \Omega$ be the vertex on the path $\sigma_i$. Therefore,
\begin{equation}
\label{Theo_prof_2}
\bar{\x}_{\sigma_i^{2q_i}} = \bO \bar{\x}_{\sigma_i^{2q_i}}+\t.
\end{equation}
If $k=1$, then $\bO=-\bI_d$ and $\t= 2\bar{\x}_{\sigma_i^{2q_i}}$ satisfy \eqref{Theo_prof_2}, and hence the equations in \eqref{prob7}. On the other hand, if $2\leq k \leq d$, then we have $k$ equations similar to \eqref{Theo_prof_2}, one for each path. We eliminate $\t$ by subtracting the equations corresponding to  $2 \leq i \leq k$ from the equation corresponding to $i=1$. This gives us
\begin{equation*}
	\label{Theo_prof_3}
	\bO(\bar{\x}_{\sigma_i^{2q_i}} - \bar{\x}_{\sigma_1^{2q_1}}) = \bar{\x}_{\sigma_i^{2q_i}}-\bar{\x}_{\sigma_1^{2q_1}} \qquad (2 \leq i \leq k).
\end{equation*}
We collect this into the fixed-point equation $\bO \bX = \bX$, where  
\begin{equation*}
	\bX = \left[  \bar{\x}_{\sigma_2^{2q_2}} - \bar{\x}_{\sigma_1^{2q_1}} \ \cdots \ \bar{\x}_{\sigma_k^{2q_k}} - \bar{\x}_{\sigma_1^{2q_1}}\right] \in \mathbb{R}^{d \times (k-1)}.
\end{equation*}
Now, if we assume that $k \leq d$, then we can find $\bO \neq \bI_d$ such that $\bO \bX = \bX$. In particular, we can find $\bO$ that acts as an identity transform on the space spanned by the columns of $\bX$, and as a non-trivial rotation on the orthogonal complement of this space. We set $\t$ using \eqref{Theo_prof_2} for this choice of $\bO$. One can  verify that the above choice of $(\bO,\t) \neq (\bI_d,\bm{0})$ satisfies the equations in \eqref{prob7}. This concludes the proof of Theorem \ref{THEOREM}.

\subsection{Experiments}

In this section, we report some additional numerical results to demonstrate the performance of the proposed algorithm.

\noindent \textbf{\underline{Experiment 6}:} To study the effect of the number of anchors on the performance, we consider a random geometric graph (RGG) on $[-0.5,0.5]^2$ consisting of $500$ sensors. The sensing radius $r$ is set as $0.17$. We plot the ANE as a function of the number of anchors $K$ for different noise levels $\eta = 0.01, 0.05$ and $0.1$. The ANE is averaged  over $100$ noise realizations. The results are reported in Figure \ref{KvsANE}. We notice that the ANE  falls off with increase in $K$, and saturates beyond a certain $K$.\newline

\begin{figure}[!htb]
	\centering
	\includegraphics[width= 0.7\linewidth]{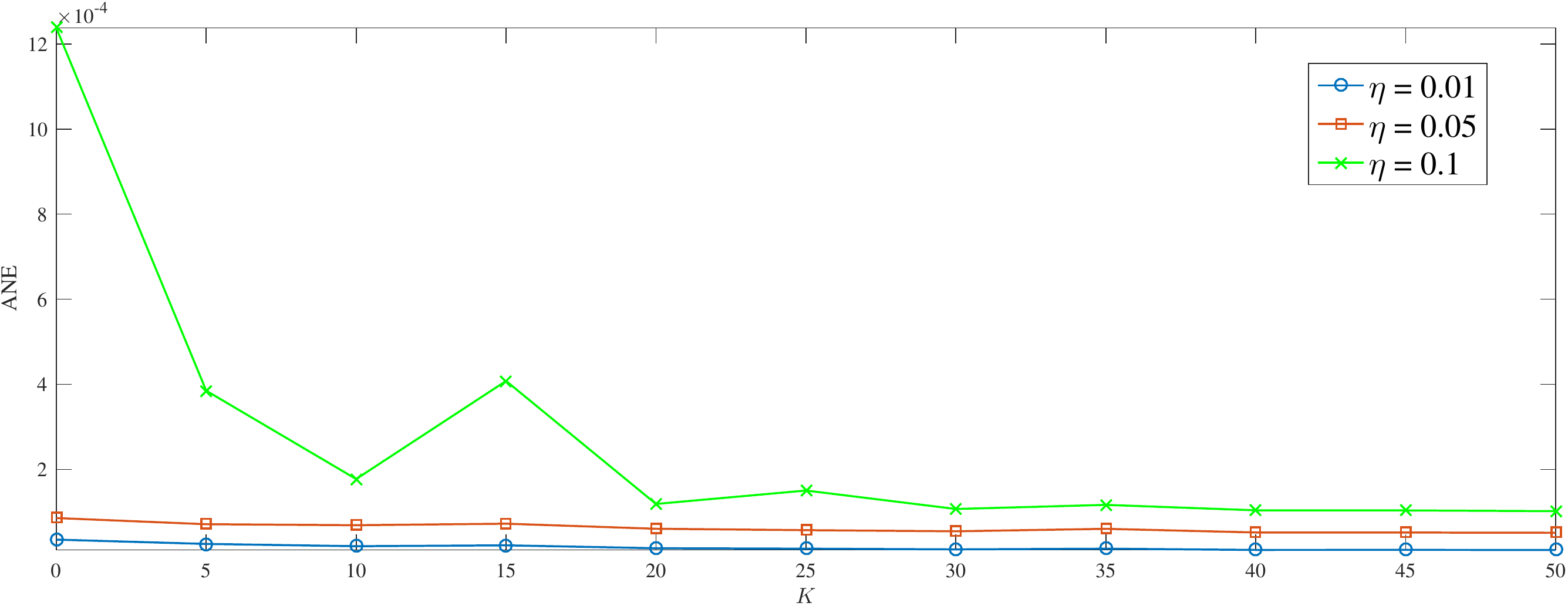}
	\caption{For a RGG with $N = 500$ and $r = 0.17$, the ANE is plotted as a function of $K$ at different noise levels $\eta = 0.01,0.05,0.1$.}
	\label{KvsANE}
\end{figure}

\noindent \textbf{\underline{Experiment 7}:} We compare the localization accuracy of the proposed method with \texttt{PLACEMENT} \cite{ADPV2012} on RGGs. The results are reported in Table \ref{Compare3}. We notice that for both clean and noisy measurements, the proposed method performs better than \texttt{PLACEMENT}.\newline

\begin{table}[!htb]
\caption{Comparison of the  proposed method with \texttt{PLACEMENT} \cite{ADPV2012} on RGGs.}
\label{Compare3}
\centering
{
\begin{tabular}{cccc|c|c|}
\cline{5-6}   &                                             &                                              &          & \multicolumn{2}{c|}{Accuracy (ANE)}                                                 \\ \hline
\multicolumn{1}{|c|}{$N$}                   & \multicolumn{1}{c|}{$K$}                  & \multicolumn{1}{c|}{$r$}                     & $\eta$ & Proposed & \texttt{PLACEMENT} \cite{ADPV2012} \\ \hline
\multicolumn{1}{|c|}{\multirow{2}{*}{$1000$}}   & \multicolumn{1}{c|}{\multirow{2}{*}{$20$}}   & \multicolumn{1}{c|}{\multirow{2}{*}{$0.12$}} & $0$      & $2.7\mbox{e-}12$ & $5.5\mbox{e-}6$  \\ \cline{4-6} 
\multicolumn{1}{|c|}{}                        & \multicolumn{1}{c|}{}                       & \multicolumn{1}{c|}{}                        & $0.01$    & $2.7\mbox{e-}3$  & $1.4$ \\ \hline
\multicolumn{1}{|c|}{\multirow{2}{*}{$2000$}}  & \multicolumn{1}{c|}{\multirow{2}{*}{$20$}}  & \multicolumn{1}{c|}{\multirow{2}{*}{$0.09$}} & $0$      & $1.5\mbox{e-}12$    & $3.7\mbox{e-}6$ \\ \cline{4-6} 
\multicolumn{1}{|c|}{}                        & \multicolumn{1}{c|}{}                       & \multicolumn{1}{c|}{}                        & $0.01$    & $8.7\mbox{e-}4$   & $1.4$ \\ \hline
\multicolumn{1}{|c|}{\multirow{2}{*}{$4000$}}   & \multicolumn{1}{c|}{\multirow{2}{*}{$20$}}   & \multicolumn{1}{c|}{\multirow{2}{*}{$0.06$}} & $0$      & $1.2\mbox{e-}10$  & $3.1\mbox{e-}2$   \\ \cline{4-6} 
\multicolumn{1}{|c|}{}                        & \multicolumn{1}{c|}{}                       & \multicolumn{1}{c|}{}                        & $0.01$    & $3.6\mbox{e-}3$   & $1.4$ \\ \hline
\multicolumn{1}{|c|}{\multirow{2}{*}{$10000$}}   & \multicolumn{1}{c|}{\multirow{2}{*}{$20$}}   & \multicolumn{1}{c|}{\multirow{2}{*}{$0.04$}} & $0$      & $6.7\mbox{e-}11$  & $1\mbox{e-}2$ \\ \cline{4-6} 
\multicolumn{1}{|c|}{}                        & \multicolumn{1}{c|}{}                       & \multicolumn{1}{c|}{}                        & $0.01$    & $1.4\mbox{e-}3$     & $1.4$ \\ \hline
\multicolumn{1}{|c|}{\multirow{2}{*}{$15000$}}  & \multicolumn{1}{c|}{\multirow{2}{*}{$20$}}  & \multicolumn{1}{c|}{\multirow{2}{*}{$0.03$}} & $0$       & $6.7\mbox{e-}10$    & $1.4$ \\ \cline{4-6} 
\multicolumn{1}{|c|}{}                        & \multicolumn{1}{c|}{}                       & \multicolumn{1}{c|}{}                        & $0.01$    & $1.9\mbox{e-}1$    & $1.4$ \\ \hline
\end{tabular}
}
\end{table}

\noindent \textbf{\underline{Experiment 8}:} We consider a RGG with $200$ sensors. For a fair comparison with \texttt{SNLDR} \cite{SXG2015}, we placed 4 anchors at $(\pm 0.5, \pm 0.5)$ (so that the sensors are guaranteed to be in the convex hull of the anchors). We set $r=0.28$ and $\eta = 0.1$. The reconstructions are compared in Figure \ref{ANE} along with the corresponding ANEs.\newline

\begin{figure}[!htb]
	\centering
	\subfloat[Proposed.]{\includegraphics[width= 0.24 \linewidth]{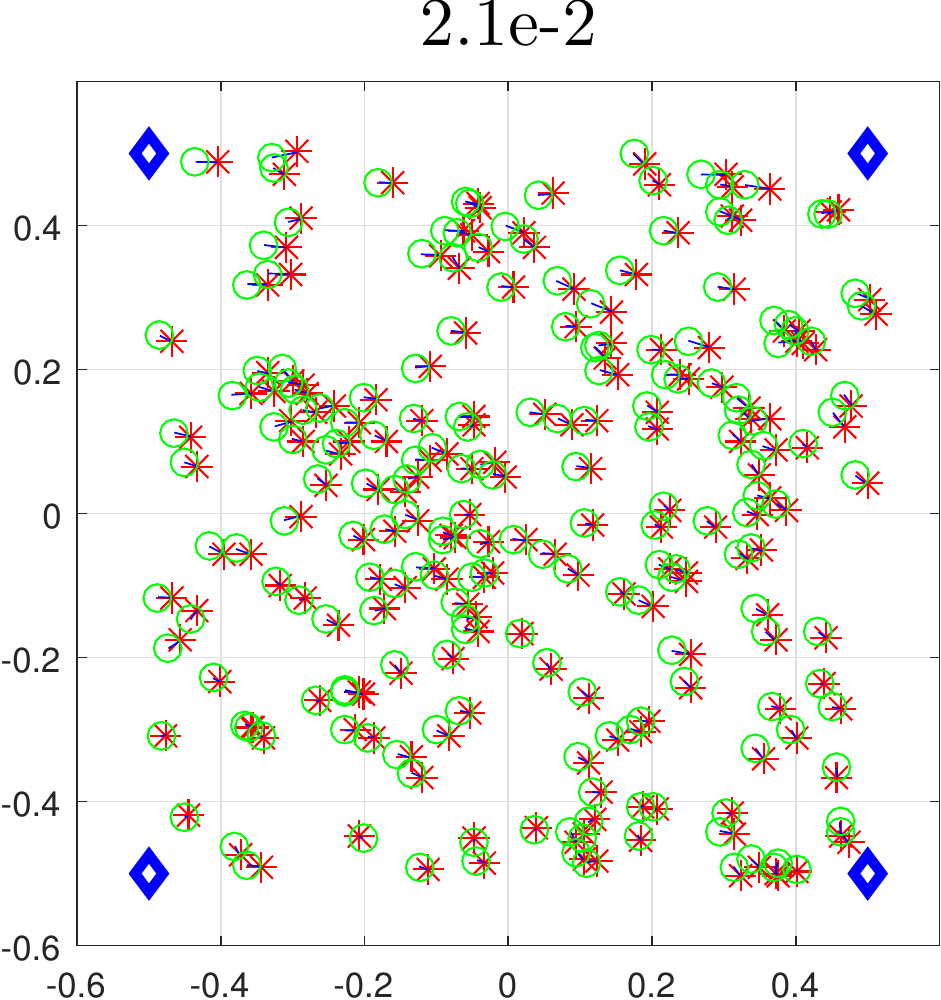}} \hspace{-0.5mm}
	\subfloat[SNLSDP.]{\includegraphics[width= 0.24 \linewidth]{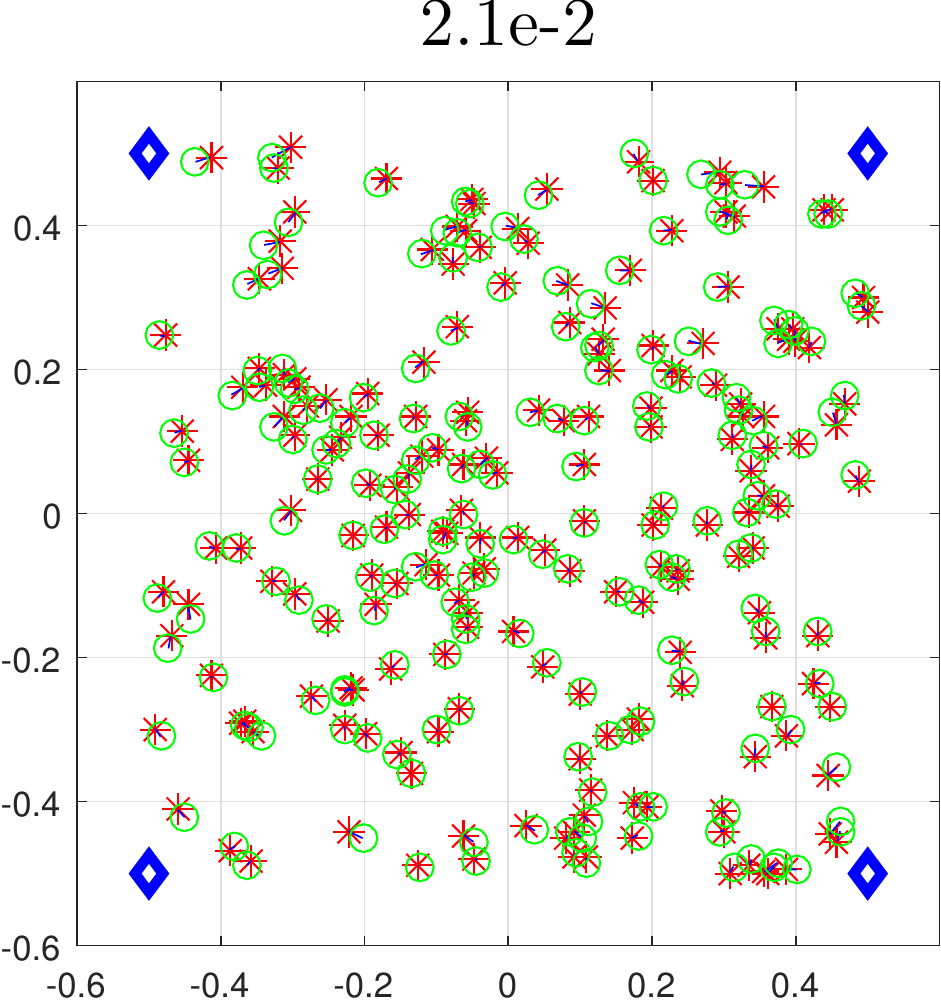}} \hspace{-0.5mm}
	\subfloat[ESDP.]{\includegraphics[width= 0.24 \linewidth]{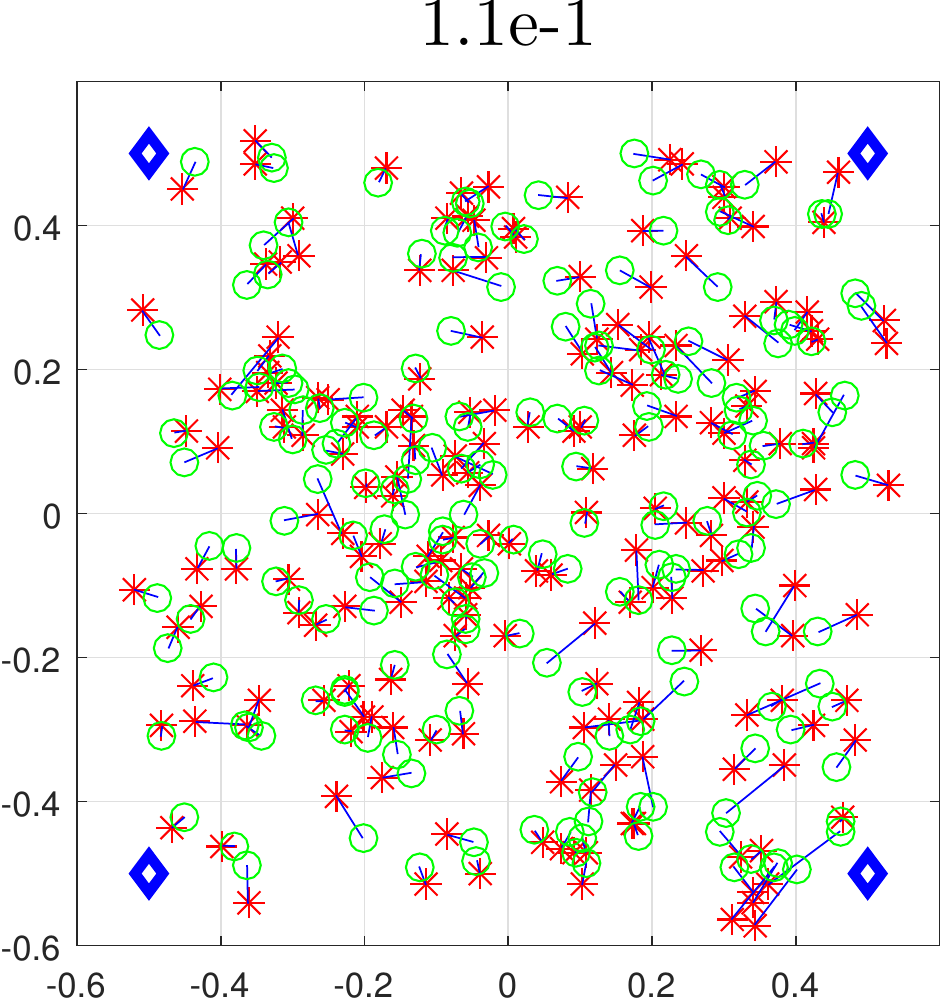}} \hspace{0.2mm}
	\subfloat[SNLDR.]{\includegraphics[width= 0.24 \linewidth]{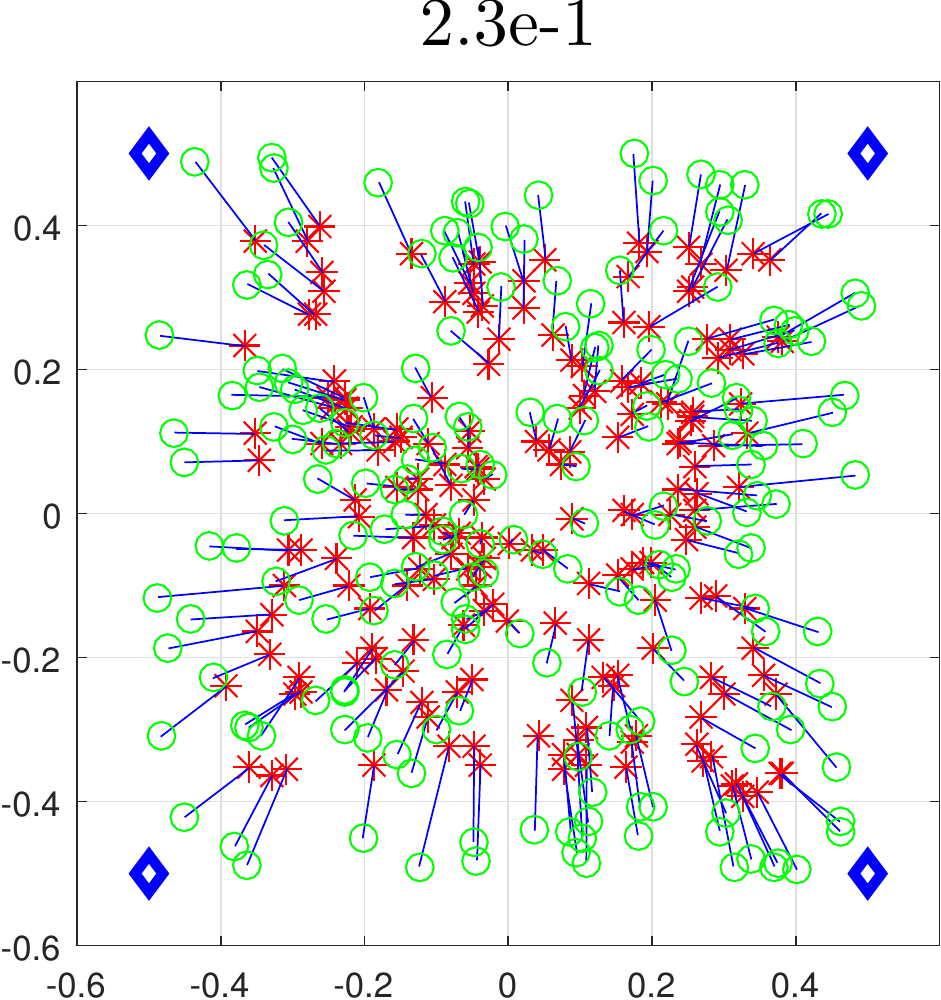}}
	\caption{Comparison of the proposed algorithm with  \cite{BLTYW2006,WZYB2008,SXG2015}. We placed $4$ anchors  at $(\pm 0.5, \pm 0.5)$. The parameters for the RGG are $N = 200, r=0.28$ and noise level was set as $\eta =0.1$. The ANE for each algorithm is mentioned at the top of the plot. Green circles (\textcolor{green}{$\circ$}) denote original sensor locations, red stars (\textcolor{red}{$\star$}) denote estimated locations, and blue diamonds (\textcolor{blue}{$\Diamond$}) denote anchor locations.}
	\label{ANE}
\end{figure}

\noindent \textbf{\underline{Experiment 9}:} We repeat Experiment 3 with $500$ sensors and $10$ anchors (with $4$ of them placed at  $(\pm 0.5, \pm 0.5)$). The localizations obtained using \texttt{ESDP} \cite{WZYB2008}, \texttt{SNLDR} \cite{SXG2015}, \texttt{SNLSDP} \cite{BLTYW2006} and the proposed method are shown in Figure \ref{ANE_deg}. For this instance of RGG, the average node degree is $13.2$ and the minimum node degree is $3$. We note that for both Experiments 3 and 4, the ANE for the proposed method is the least.

\begin{figure}[!htb]
	\centering
	\subfloat[Proposed.]{\includegraphics[width= 0.24 \linewidth]{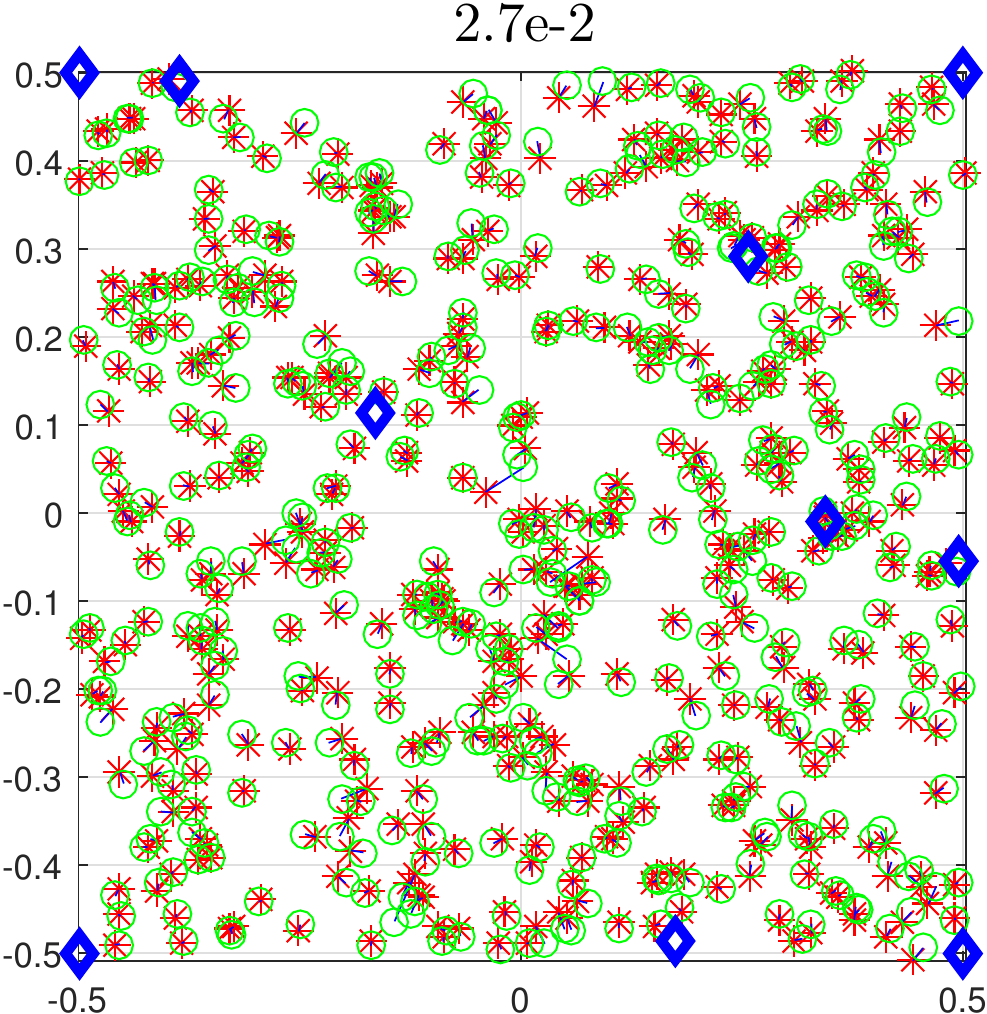}} \hspace{-0.5mm}
	\subfloat[SNLSDP.]{\includegraphics[width= 0.24 \linewidth]{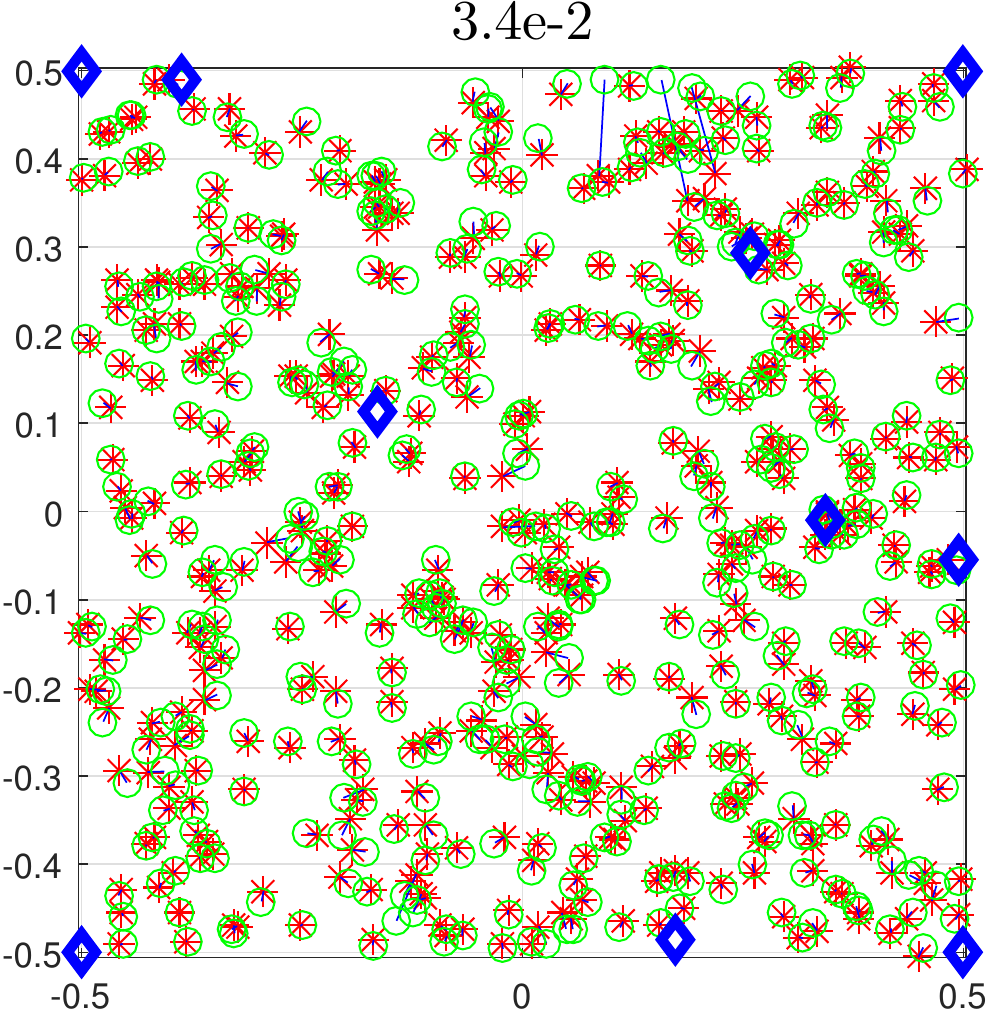}} \hspace{-0.5mm}
	\subfloat[ESDP.]{\includegraphics[width= 0.24 \linewidth]{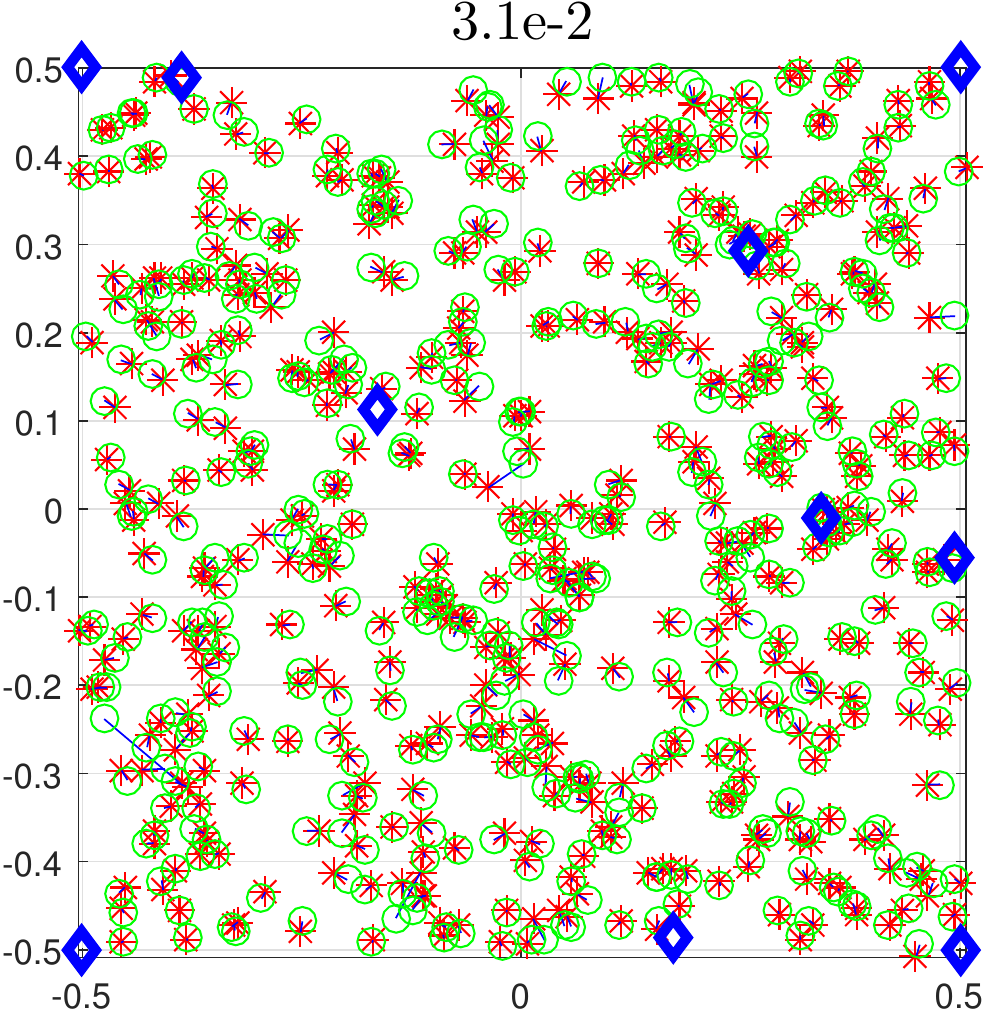}} \hspace{0.2mm}
	\subfloat[SNLDR.]{\includegraphics[width= 0.24 \linewidth]{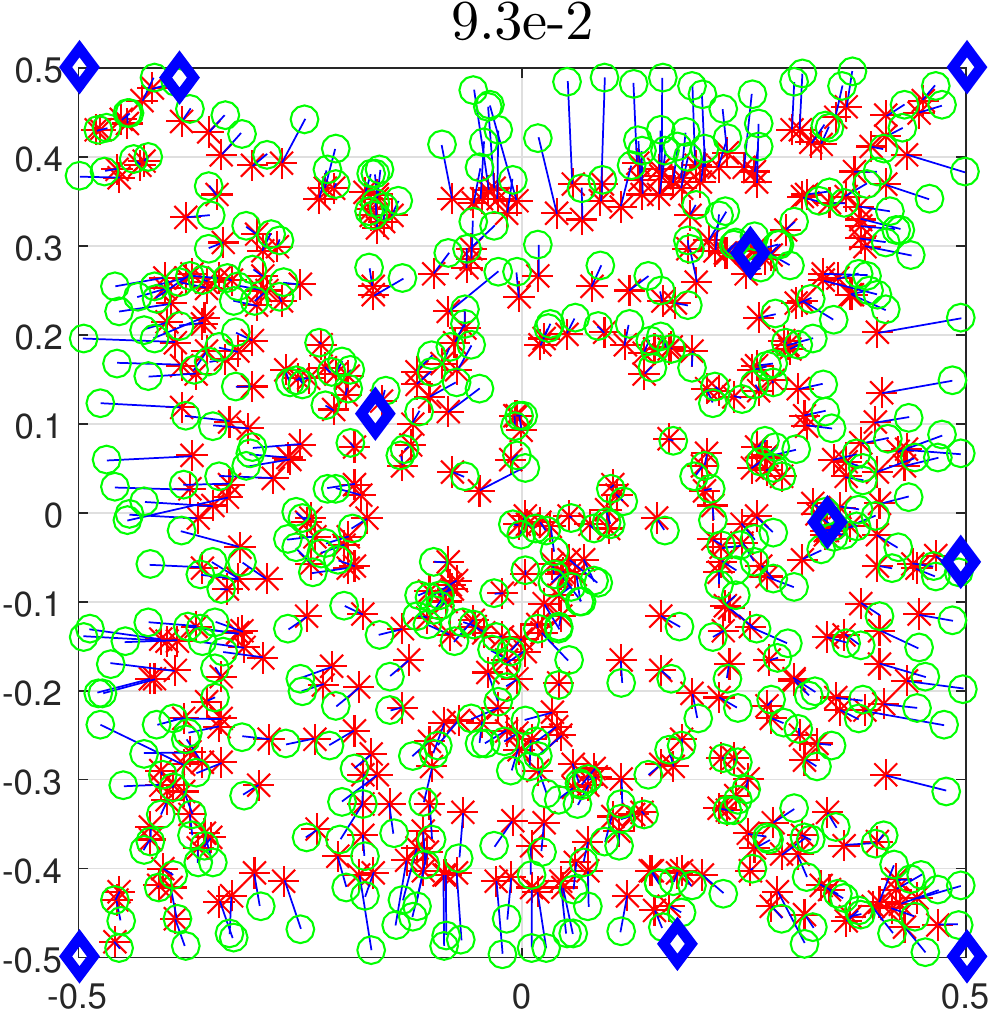}}
	\caption{Comparison of the proposed algorithm with \cite{BLTYW2006,WZYB2008,SXG2015}. We have used $10$ anchors of which $4$ of them are placed at $(\pm 0.5, \pm 0.5)$. The parameters used are $N = 500$ and $\eta =0.1$, and the average node degree is $13.2$.}
		\label{ANE_deg}
\end{figure}

\section*{Acknowledgements}
The authors wish to thank the editor and the anonymous reviewers for their thoughtful comments and suggestions. The authors also wish to thank Nicolas Gillis, Andrea Simonetto, Claudia Soares, and Arvind Agarwal for useful discussions and for providing the MATLAB code of their algorithms. 
\bibliographystyle{IEEEtran}
\bibliography{citations}
\end{document}